\documentclass[12pt]{article}

\usepackage[margin=1.25in]{geometry}

\usepackage{amssymb,amsmath}
\usepackage{amsthm}
\usepackage{hyperref}
\usepackage{mathrsfs}
\usepackage{textcomp}
\hypersetup{
    colorlinks,%
    citecolor=red,%
    filecolor=black,%
    linkcolor=blue,%
    urlcolor=black
}

\usepackage{verbatim}

\usepackage{sectsty}

\usepackage{esint}

\usepackage[pagewise]{lineno} 

\usepackage{tikz}
    
\makeatletter

\newdimen\bibspace
\renewenvironment{thebibliography}[1]{%
 \section*{\refname 
       \@mkboth{\MakeUppercase\refname}{\MakeUppercase\refname}}%
     \list{\@biblabel{\@arabic\c@enumiv}}%
          {\settowidth\labelwidth{\@biblabel{#1}}%
           \leftmargin\labelwidth
           \advance\leftmargin\labelsep
           \itemsep\bibspace
           \parsep\z@skip     %
           \@openbib@code
           \usecounter{enumiv}%
           \let\p@enumiv\@empty
           \renewcommand\theenumiv{\@arabic\c@enumiv}}%
     \sloppy\clubpenalty4000\widowpenalty4000%
     \sfcode`\.\@m}
    {\def\@noitemerr
      {\@latex@warning{Empty `thebibliography' environment}}%
     \endlist}

\makeatother

\makeatletter

\newtheorem{thm}{Theorem}[section]
\newtheorem{lem}[thm]{Lemma}

\newtheorem{defn}[thm]{Definition}

\newtheorem{cor}[thm]{Corollary}
\newtheorem{rem}[thm]{Remark}

\def\XXint#1#2#3{{\setbox0=\hbox{$#1{#2#3}{\int}$}
  \vcenter{\hbox{$#2#3$}}\kern-.5\wd0}}

\newcommand{\om}{\Omega}                \newcommand{\pa}{\partial}
           \newcommand{\ud}{\mathrm{d}}
\newcommand{\be}{\begin{equation}}      \newcommand{\ee}{\end{equation}}

\newcommand{\R}{\mathbb{R}}

\begin{document}

\title{\textbf{Schauder estimates for the boundary diffusion equation}
\bigskip}

\author{\large  Xuzhou Yang }
\date{}

\maketitle

\begin{abstract}
In this paper, a certain type of linear boundary diffusion equation is studied. Such equation is crucial in the research of a non-linear boundary diffusion problem which was originated from the boundary heat control problem and the boundary Yamabe flow. Such equation is also analogous to a fractional porous medium equation. For compatible-interior-energy solutions, with the help of two types of test functions, we find a way to localize the non-local boundary diffusion equation, and a theory of $C^{1+\alpha}$-type Schauder estimates of weak compatible-interior-energy solutions are established in the linear boundary diffusion equation, through iteration to integrals both on the interior and the boundary at the same time. As applications, Schauder estimates of this new $C^{1+\alpha}$-type solution to the related boundary diffusion equations are generated, which eventually provide a proof of short time existence of $C^{1+\alpha}$ compatible-interior-energy solution with smooth positive initial data and inhomogeneous boundary term.
\end{abstract}


\smallskip


\bigskip

\tableofcontents

\section{Introduction}\label{sec:intro}

Let $ \Omega \subset \R^n$, $n \geq 2 $,  be a bounded domain with $C^{1+\alpha}$ boundary. We are concerned with Schauder estimates for the positive classical solution of a boundary diffusion equation
\begin{equation}\label{eq:boundary diffusion}
 \left\{
\begin{aligned}
  \mbox{div} (a \nabla v)  &=   0  \quad  \mbox{in }     \Omega\times (0,T),     \\
   \partial_{t}v+a \partial_{\nu}v + bv &=  f    \quad  \mbox{on }     \partial \Omega\times (0,T),
\end{aligned}
\right.
\end{equation}
where $\Delta $ is the Laplace operator with respect to the spatial variable $x\in \om$, $\nu$ is the unit outer normal to $\pa \om$, $\pa_{\nu}$ is the outer normal derivative. $a(x,t)\in C^{\alpha}(\bar{\Omega}\times(0,T)) \bigcap  C^{\alpha}(\pa {\Omega}\times(0,T))$ and $b(x,t), f(x,t) \in C^{\alpha}\left(\partial \Omega\times (0,T)\right)$. The boundary term $\pa_tv + \pa_\nu v$ represents the process of thermal contact on the boundary, see \cite{1975Crank}. 
 
Furthermore, taking
$$
v=u^p, a= \frac{1}{p}v^{\frac{1-p}{p}},
$$
then the equation \eqref{eq:boundary diffusion} can be treated as the linearized equation of a non-linear boundary diffusion equation
\begin{equation}\label{eq:nonlinear boundary diffusion}
 \left\{
\begin{aligned}
  \Delta u  &=   0  \quad  \mbox{in }     \Omega\times (0,\infty),     \\
   \partial_{t}u^{p}&=   -  \partial_{\nu}u  -  b u^p     \quad  \mbox{on }     \partial \Omega\times (0,\infty).
\end{aligned}
\right.
\end{equation}
This non-linear boundary  diffusion equation \eqref{eq:nonlinear boundary diffusion} is a boundary equation of porous medium type and it arises from the boundary heat control problems which are very important in thermodynamics. When $n\ge 3$, $p=\frac{n}{n-2}$ and $b=-\frac{n-2}{2(n-1)}H$ with $H$  being the mean curvature of $\pa \om$, it is the unnormalized boundary Yamabe flow; see  Brendle \cite{brendle2002generalization} and Almaraz \cite{Alm}. 
With initial condition
\be \label{eq: boundary diffusion initial condition u_0}
 u  =   u_0 >0 \quad  \mbox{on }     \partial \Omega\times \{ t=0 \} , u_0 \in C^{\infty}(\pa \om),
\ee
\eqref{eq:nonlinear boundary diffusion}-\eqref{eq: boundary diffusion initial condition u_0} formulate a non-linear diffusion problem. In 2002, in order to search for a conformal metric with vanishing scalar curvature in the interior and constant mean curvature at the boundary, Brendle \cite{brendle2002generalization} analyzed the long time behavior of the solution of the boundary Yamabe problem \eqref{eq:nonlinear boundary diffusion}-\eqref{eq: boundary diffusion initial condition u_0} when $p=\frac{n}{n-2}$. Recently in 2024, the existence, lower and upper bounds of the positive smooth solution of the non-linear boundary diffusion problem \eqref{eq:nonlinear boundary diffusion}-\eqref{eq: boundary diffusion initial condition u_0} were studied by Jin-Xiong-Yang \cite{2024JXY} and the long time behavior was  classified and illustrated.  Especially, they proved the short time existence of smooth positive solution to the problem \eqref{eq:nonlinear boundary diffusion}-\eqref{eq: boundary diffusion initial condition u_0} through Sobolev theory and implicit function theorem. The long time existence was obtained by extending such solution whenever they were positive, and the asymptotics of the positive smooth solution of  \eqref{eq:nonlinear boundary diffusion}-\eqref{eq: boundary diffusion initial condition u_0} for both $0<p<1$ and $1<p<\frac{n}{n-2} $ are obtained with sharp convergence rates.

Early in 1972, the physical model of boundary heat control was carried out by Duvaut-Lions \cite{DL1972}. Later in 2010, Athanasopoulos-Caffarelli \cite{2010Continuity} studied the boundary heat control model  
\begin{equation}\label{eq:BPME}
 \left\{
\begin{aligned}
& \alpha\partial_{t}u  =\Delta u \quad                           \mbox{in }   \Omega \times (0,T],\\
&-\partial_\nu u  \in  \partial_{t}(\beta(u)) \quad    \mbox{on }   \Gamma \times  (0,T] , \\
&u              =  0  \quad   \mbox{on } \left( \partial \Omega - \Gamma \right) \times [0,T],\\
&u(x,0)=u_0(x)    \quad      \mbox{on }  \Omega.
\end{aligned}
\right.
\end{equation}
They proved the H\"older continuity of the weak solution to \eqref{eq:BPME} for $\beta(u)=u^q,0<q<1$, and the sign ``$\in$" be ``$=$",
which is an approximation to the Stefan problem described in \cite{DL1972}. The mathematical model \eqref{eq:BPME} can be used to describe the temperature in ice-water mixture. In \cite{2010Continuity}, the authors used some barrier functions that are exquisitely built to employ the comparison principle. Then together with the aid of DeGiorgi type iteration, they obtained $L^{\infty}$ estimate and continuity (with proper modulus) of the weak solution of \eqref{eq:BPME} with $\beta(u)$ satisfying some certain growth conditions.

The boundary diffusion problem \eqref{eq:nonlinear boundary diffusion}-\eqref{eq: boundary diffusion initial condition u_0} is like a ``lite" version of the above heat control problem \eqref{eq:BPME}. Under this circumstance, we can view the boundary diffusion equation \eqref{eq:boundary diffusion} as a fractional diffusion equation on a compact manifold $\pa \Omega$. If we denote $\mathscr{B}:  H^{\frac12}(\partial \Omega) \rightarrow H^{-\frac12}(\partial \Omega)$ as the Dirichlet to Neumann boundary operator, that is, for any $u \in H^{\frac12}(\partial \Omega) $,
\begin{equation}\label{eq:DN}
\mathscr{B}u:= \frac{\partial}{\partial \nu}U\Big{|}_{\partial \Omega},
\end{equation}
where $U$ satisfies $\Delta U=0$ in $\Omega$ and $U=u$ on $\pa \om$, then \eqref{eq:boundary diffusion} can be rewritten as
\begin{equation}\label{eq:DtN diffusion}
   \partial_{t}v+ \mathscr{B} (v^m)       + bv =  f   
      \quad  \mbox{on }     \partial \Omega\times (0,T),
\end{equation}
when $a=mv^{m-1}$. Therefore, it follows from Chang-Gonz\'alez \cite{ChangG} that the equation \eqref{eq:boundary diffusion} can be viewed as a linearization of a type of fractional porous medium equation  with a non-local operator $\mathscr{B} $ on the manifold $\partial\Omega$ of  principal symbol the same as the $1/2-$Laplace operator on $\partial\Omega$. 
If 
$$
         \ a= mv^{m-1},
$$
and 
$$
\Omega=\R^n_+=\{x=(x',x_n): x_n>0\}, \  b\equiv 0,\  f\equiv 0,
$$ 
then \eqref{eq:boundary diffusion} becomes the fractional porous medium equation on the Euclidean space 
\begin{equation}\label{eq:fpme1}
   \partial_{t}v +  (-\Delta)^{\sigma}v^m  =  0    \quad  \mbox{on }     \R^{n-1}\times (0,\infty)
\end{equation}
with $\sigma=1/2$. This equation \eqref{eq:fpme1} with $\sigma=1/2$ had been systematically studied by de Pablo-Quir\'os-Rodr\'iguez-V\'azquez \cite{PQRV}, which was later extended to the general fractional Laplace operator $(-\Delta)^{\sigma}$ for all $\sigma\in (0,1)$ by themselves in \cite{PQRV2}. V\'azquez \cite{Vaz}  established the existence, uniqueness of the fundamental solutions to \eqref{eq:fpme1}, and obtained some asymptotics properties. When $m$ is the critical Sobolev exponent, Jin-Xiong \cite{JX2014} studied the asymtotic behavior of the solutions to \eqref{eq:fpme1}. The weighted global integral estimates for \eqref{eq:fpme1} were computed by Bonforte-V\'azquez \cite{BV2014}. These estimates were improved by V\'azquez-Volzone \cite{JL2013Optimal}. V\'azquez-de Pablo-Quir\'os-Rodr\'iguez \cite{VPQR} proved the weak solutions of \eqref{eq:fpme1} are classical for all positive time.

If we take 
\begin{equation*}
u^p=u^{\frac{1}{m}}=v,       \ a= mv^{m-1} , \ f\equiv 0,
\end{equation*} 
then \eqref{eq:boundary diffusion} becomes the non-linear diffusion equation \eqref{eq:nonlinear boundary diffusion}. Notably, the proof of the short time existence of smooth positive solution to the problem \eqref{eq:nonlinear boundary diffusion}-\eqref{eq: boundary diffusion initial condition u_0} was guided by the observation that 
\be\label{eq:positive bounds}  
c_0 \le u \le C_0
\ee
for some positive constants $c_0$ and $C_0$ on a small period of time. Furthermore, these positive lower and upper bounds were well established in Lemma \textcolor{blue}{3.1} of \cite{2024JXY}.
The similar idea was used early in 2002, Lemma \textcolor{blue}{3.4} of \cite{brendle2002generalization}, where Brendle proved the short time existence of smooth positive solution for boundary diffusion equation \eqref{eq:boundary diffusion} with initial datum
\be \label{eq: boundary diffusion initial condition v_0}
 v  =   v_0 := u_0^{\frac{n}{n-2}}>0 \quad  \mbox{on }     \partial \Omega\times \{ t=0 \} ,
\ee
and the relation that
$$
 u^{\frac{n}{n-2}}=v , \   a= (n-1)v^{-\frac{2}{n}},\ b=\frac{n}{2}H_0 v^{-\frac{2}{n}}  , \ f\equiv 0,
$$
where $H$ is the mean curvature of $\pa \Omega$ and $\bar{H}:=\int_{\pa \Omega} H dS / \int_{\pa \Omega}  dS$ is the average value of the mean curvature $H$.

Inspired by \eqref{eq:positive bounds}, it naturally drives us to seek Schauder type estimates of classical solution of \eqref{eq:boundary diffusion} under condition
\begin{equation}\label{eq:bound of coefficients}
\begin{array}{cc}
& 1/\Lambda   \le a(x,t)  \le \Lambda,\\
&|a(x,t)|_{C^{\alpha}(\bar{ \Omega}  \times [0,T])}+|b(x,t)|_{L^{\infty}(\pa \Omega \times [0,T])}+|f(x,t)|_{C^{\alpha}(\pa \Omega \times [0,T])} \le \Lambda, \\
&|\pa \om|_{C^{1+\alpha}} \le \Lambda
\end{array}
\end{equation}
where $\Lambda >1$ is a constant. 

For the global equation \eqref{eq:DtN diffusion} with a non-local operator, one can get estimates by using the ellipticity of the operator $\mathscr{B} $ globally on $\partial \Omega$.  In \cite{2024JXY}, we obtained $L^{\infty}$ estimate of the solutions to \eqref{eq:nonlinear boundary diffusion} by cutting off only on time and applying Moser's iteration. For Schauder type estimates, it is not applicable only cutting off on time, since we need to consider the equation locally on space as well. Locally investigating the boundary diffusion equation \eqref{eq:boundary diffusion}, we will lose the structure of the non-local operator
$$
\| \mathscr{B} u  \|_{L^2{(\pa \Omega})}   \sim  \| \nabla u  \|_{L^2{(\pa \Omega})} .
$$ 
This non-local structure is crucial for applying Sobolev theory and promoting regularity of weak solutions, see \cite{brendle2002generalization}. 

On the other hand, in the local configuration, we can not lift regularity by the non-local operator $\mathscr{B}$, making it insufficient to find a proper weak solution through traditional energy estimate. Therefore, it is necessary to provide a perspective to localize such non-local boundary diffusion equation, well posing the local version of the global equation \eqref{eq:boundary diffusion} which is generated by cutting off on space, finding corresponding Caccioppoli type inequalities for suitable weak solutions as preparation for the Schauder type estimates. Lemma \ref{lem:homogeneous-caccioppoli-1.5}-\ref{lem:homogeneous-iteration} are consequences from dealing with these issues.

Unlike the general parabolic system, this boundary diffusion equation \eqref{eq:boundary diffusion} has two different diffusion on  the interior domain and  the boundary. When we want to engage into Schauder estimate to such problem, we encounter much difficulty due to the scaling difference between the equations inside the domain and on the boundary. Since the dimensions of the interior domain and the boundary are different, it annoys us when we calculate the average integrals. And we find it neither sufficient to merely estimate the integral on the boundary, nor to compute the integral on the interior domain. Luckily, we find a quantity in which  the local integrals both on the interior domain and the boundary are put together, ensuring the successful launch of Schauder-type iteration to such two types of integrals at the same time to reach the desired estimates in Theorem \ref{thm:estimate campanto norm}.

As an important application of the Schauder theory, the Schauder estimates of  \eqref{eq:boundary diffusion} are crucial for understanding the existence theory to the positive classical solution to the non-linear problem \eqref{eq:nonlinear boundary diffusion}-\eqref{eq: boundary diffusion initial condition u_0}. With Schauder estiamtes of \eqref{eq:boundary diffusion} at hand, we  prove the short time existence of $C^{1+\alpha}$-type solution of \eqref{eq:nonlinear boundary diffusion}-\eqref{eq: boundary diffusion initial condition u_0} by Leray-Schauder theory in Theorem \ref{thm:nonlinear boundary diffusion non-homogeneous short time existence }.

To establish Schauder type estimates for the boundary diffusion equation \eqref{eq:boundary diffusion} in a general local setting, we locally flatten the boundary $\pa \Omega$. Since the boundary $\pa \Omega$ is of class $C^{1+\alpha}$, there is a $C^{1+\alpha}$ function $\varphi: {\R}^{n-1} \rightarrow \R$ such that 
$$
B(x^0,R_0)\cap\Omega=\{x\in B(x^0,R_0) | x_n > \varphi(x_1,\cdots,x_{n-1} )   \}
$$
for any $x^{0}\in \pa \Omega$. We change variables
\begin{align*}
& y=\Phi(x)= ( x_1-x^{0}_1,\cdots,x_{n-1}-x^{0}_{n-1},x_n-\varphi(x_1,\cdots,x_{n-1}) ), \\
&\tau = t-t^0
\end{align*}
for some $(x^{0},t^{0})\in \pa \Omega \times (0,T)$.
Correspondingly, it gives
\begin{align*}
& x=\Psi(y)=(y_1+x^{0}_1,\cdots,y_{n-1}+x^{0}_{n-1},y_n+\varphi(y_1+x^{0}_1,\cdots,y_{n-1}+x^{0}_{n-1})), \\
&\ t= \tau+t^0.
\end{align*}
We choose $\tilde{R}$ small such that 
$$
B^{+}_{\bar{R}}=\{ y\in B(\tilde{R},0)| y_n >0    \} \subset  \Phi(\Omega \cap B(x^0,R_0)).
$$
Taking
$$
R = \min  \{ \tilde{R}, t^0, T-t^0 \},
$$

$$
\tilde{v}(y,\tau)=v( \Psi(y), \tau + t^0  ),
$$
and using the Einstein summation convention the sum of the indexes is from $1$ to $n$, we derive the equation of $\tilde{v}$ as
\begin{equation}\label{eq:boundary diffusion local}
 \left\{
\begin{aligned}
  \pa_i(\tilde{a} a_{ij}(y)\pa_j\tilde{v})  &=   0  \quad  \mbox{in }     B^+_R \times I_R,     \\
   \partial_{\tau}\tilde{v}-\tilde{a} \frac{1}{\sqrt{  |D_{x'} \varphi|^2+1 }} a_{nj}(y)\pa_j\tilde{v} + \tilde{b}\tilde{v} &=  \tilde{f}    \quad  \mbox{on }     \pa ' B^+_R \times I_R
\end{aligned}
\right.
\end{equation}
where we denote  $y=(y',y_n), \pa ' B^+_R=B_R\cap\{y_n=0 \} ,I_R=(-R,R)$, correspondingly
$\tilde{a}(y,\tau)=a(x,t),\tilde{b}(y,\tau)=b(x,t),\tilde{f}(y,\tau)=f(x,t)$ and the partial derivative $\pa_i$ is taken with respect to variable $y_i$. Here, one can compute
\begin{equation}\label{eq:flatten coefficient matrix}
a_{ij}(y)=\begin{pmatrix}
1 & 0 & \cdots & 0 & -\frac{\pa \varphi}{\pa x_1} \\
0 & \ddots &  & 0 & \vdots \\
\vdots &  & \ddots & 0 & \vdots \\
0& \cdots & 0 & 1 & -\frac{\pa \varphi}{\pa x_{n-1}}\\
-\frac{\pa \varphi}{\pa x_1}& \cdots &\cdots & -\frac{\pa \varphi}{\pa x_{n-1}} & |D_{x'} \varphi|^2+1
\end{pmatrix},
\end{equation}
with $\varphi=\varphi(\Psi(y)_1,\cdots,\Psi(y)_{n-1})$. Locally flattening the boundary, we find an elliptic system with divergence form in the equation \eqref{eq:boundary diffusion local}. Denote
$$
\tilde{\varphi } := \frac{1}{\sqrt{  |D_{x'} \varphi|^2+1 }},
$$
and extend it onto $B_R^+$ so that it is a constant function with respect to $x_n$. We rewrite \eqref{eq:boundary diffusion local} into a more general form with variable $(x,t)$ as 
\begin{equation}\label{eq:boundary diffusion local standard}
 \left\{
\begin{aligned}
  \pa_i(a(x,t)a_{ij}(x,t)\pa_j u)  &=   0  \quad  \mbox{in }    Q^+_R= B^+_R \times I_R,     \\
   \partial_{t}u-a(x,t)\tilde{\varphi}(x)  a_{nj}(x,t)\pa_ju + bu &=  f \quad  \mbox{on }  \pa ' Q^+_R =  \pa ' B^+_R \times I_R.
\end{aligned}
\right.
\end{equation} 
Besides, the condition \eqref{eq:bound of coefficients} is converted as 
\begin{equation}\label{eq:bound of coefficients-local}
\begin{array}{cc}
& \frac{1}{\Lambda}   \le a(x,t) \le \Lambda,\\
&|a(x,t)|_{C^{\alpha}(\pa ' Q^+_R )}
+|\tilde{\varphi}(x)|_{C^{\alpha}(\pa ' B^+_R
)}+|b(x,t)|_{C^{\alpha}(\pa ' Q^+_R)} +|a_{ij}(x,t)|_{C^{\alpha}(\pa ' Q^+_R  )} \le \Lambda,\\
& \frac{1}{\Lambda}  |\xi|^2  \le a_{ij}(x,t)\xi_i \xi_j  \le \Lambda |\xi|^2 ,  \quad   \forall \ \xi \in \R^n , \\
&  a_{ij}(x,t)=a_{ji}(x,t), \quad \forall i,j \in \{1,\cdots,n\}
\end{array}
\end{equation}
for all $(x,t) \in  \pa'Q^+_R $ uniformly.

The upper and lower bounds of $a(x,t)$ in \eqref{eq:bound of coefficients-local} correspond to the ellipticity of such local equation, making \eqref{eq:boundary diffusion local standard} looks like a parabolic equation. But, in \eqref{eq:boundary diffusion local standard}, we see that it possesses the scaling distinct to the scaling of the heat equation, only giving the information of $\pa_t u$ and $\pa_n u$ on $\pa ' Q^+_R$. Another bad thing is that in the local equation, we lose the Dirichlet to Neumann boundary operator as defined in \eqref{eq:DN}, making it impossible to use the non-local operator $\mathscr{B}$ to promote the regularity. It seems that we lose some ellipticity on $\pa' Q_R^+$ when we break the global equation into pieces.

In order to well pose the local equation \eqref{eq:boundary diffusion local standard} without the global fractional diffusion structure, we propose weak compatible-interior-energy solutions in stead of the usual weak solutions as a view to locally address \eqref{eq:boundary diffusion}.

\begin{defn}\label{con:compatibility}(Compatible-Interior-Energy solution)

For  a  weak solution $u \in  H^{1}\left(\pa ' Q^+_R \right) \bigcap L^{\infty}\left(I_R:H^1(B_R^+)\right) $ of \eqref{eq:boundary diffusion local standard}, we say it is a  compatible-interior-energy solution (briefly written as CIE solution), if it additionally satisfies
\be \label{eq:compatibility}\tag{E}
 \int_{  Q^+_{R}  } |u_t|^2  \le C(n,\Lambda)\left( \int_{  Q^+_R  } |D_x  u|^2 + \frac{1}{R^2}\int_{  Q^+_R  } | u|^2 \right).
\ee

\end{defn}
We shall prove the local Schauder estimates regarding this type of solutions.

In terms of solution $u$ to \eqref{eq:boundary diffusion local standard}, $\pa_t u$ may fail to exist in $Q^+_R$ in weak sense and it only admits very weak solutions $u\in L^2\left(I_R: H^1 (  B^+_R) \right) \bigcap L^{\infty}\left(I_R:L^2(\pa' B^+_R)\right) $ if it is lack of  energy compatibility \eqref{eq:compatibility}. Thus, CIE solutions are considered in Section \ref{sec:weak solution}, in stead of the usual weak solutions.

In particular, if $u$ is a solution to \eqref{eq:boundary diffusion local standard}, then it should maintain compatible interior energy for space and time, that is
\be \label{eq:compatible eq}
u_t  \sim    a_{nj}\pa_j u +u \quad \mbox{ in } Q^+_R, \mbox{ when } f\equiv 0,
\ee
since $a_{nj}\pa_j u $ is well defined in $Q^+_R$ as well. The corresponding weak solutions should weakly satisfy \eqref{eq:compatible eq} and such relation is maintained for CIE solutions. It is noticed that in \eqref{eq:flatten coefficient matrix} and \eqref{eq:bound of coefficients-local}, all the coefficients are functions independent of $x_n$ and the norms of all the coefficients are given only on $\pa'Q^+_R$. So \eqref{eq:compatible eq} is necessary to get the estimates in $Q_R^+$.

\begin{rem}\label{rem:global compatible}
For the global equation \eqref{eq:boundary diffusion} with coefficients satisfying \eqref{eq:bound of coefficients}, one would see $u\in H^{K+1}(\pa \om \times (0,T))=W^{K+1,2}(\pa \om \times (0,T))$ if $f \in H^{K}(\pa \om \times (0,T))$. And we would expect $u(\cdot, t) \in H^{K+2}(\om)$ for $t \in (0,T)$. This global structure allow us to gain regularity on $\pa \om \times (0,T)$ both in time and space, so it is not imminent to have estimate of $\pa_t u$ in $\om \times (0,T)$, see \cite{2024JXY}. Especially, the interior profile of $\pa_t u$ can be characterized by the Green's function of $\om $ in the global case \eqref{eq:boundary diffusion}, but such representation is absent in the local case \eqref{eq:boundary diffusion local standard}.
\end{rem}

CIE solution contains information of $u_t$ inside $\om \cap B_R$ which is compatible with the information on the boundary, ensuring the existence results proved in Section \ref{sec:weak solution}, whose configuration produces an existence result of a new type of CIE solution that belongs to $C^{1+\alpha}(\pa' Q^+_R) \bigcap C^{1+\alpha}( Q^+_R).$

In Section \ref{sec:constant coefficients}, we are able to get the estimate of $\pa_t u$ in $ Q^+_R$ for CIE solutions while considering such solutions in local situation \eqref{eq:boundary diffusion local standard}. Then it is natural to use two additional test functions to find the connection between
$$
\int_{\pa' Q^+_R} |\pa_n u|^2 \mbox{ and } \int_{\pa' Q^+_R} |\nabla u|^2 ,
$$
providing us the estimate of $\int_{\pa' Q^+_R} |\nabla u|^2$.

Normally, if we use Moser's iteration to get $L^{\infty}$ estimate of the weak solution to \eqref{eq:boundary diffusion local standard}, then one only need to work on the interior term
$$
\int_{ Q^+_R} u^2
$$
since  the estimate of 
$$
\sup_{Q^+_{R/2}} u 
$$
gives the estimate of $\| u \|_{L^{\infty}(\pa'Q^+_{R/2})} $. Differently, the idea to use Campanato theory to characterize the H\"older semi norm of the derivatives of the solution $u$, is that to have the estimates on the boundary term, controlling  $\int_{\pa' Q^+_R} |\nabla u|^2$ through $\int_{\pa ' Q^+_R} u^2$. However, we are not able to control such term with $\int_{\pa ' Q^+_R} u^2$ solely. Therefore, instead working on the boundary term
$$
\int_{\pa ' Q^+_R} u^2
$$
we need to consider the integrals both on the interior and on the boundary at the same time, that is working on
$$
\int_{ Q^+_R} u^2+ \int_{\pa ' Q^+_R} u^2,
$$
when iterating for the Schauder type estimates.  Now we encounter two types of integral estimates. The first is on a part of the interior domain, the second is on a flatten part of the boundary. These two types of integrals are put together to generate the estimates, see  Section \ref{sec:Schauder}.

Here we state our Schauder estimate for the localized equation \eqref{eq:boundary diffusion local standard}.

\begin{thm}\label{thm:schauder estimate local general}
If $u \in C^{1+\alpha}(\pa' Q^+_R) \bigcap C^{1+\alpha}( Q^+_R)$ is a CIE solution to \eqref{eq:boundary diffusion local standard} with the coefficients satisfying \eqref{eq:bound of coefficients-local}, then we have Schauder type estimate
\be\label{eq:schauder estimate local general}
\begin{split}
&[u]^2_{C^{1+\alpha}(\pa ' Q^+_{R/2} )} +[u]^2_{C^{1+\alpha}( Q^+_{R/2} )}\\
&\le C(n,\Lambda,\alpha)\left( \|u\|^2_{H^1(\pa ' Q^+_R )}+\|u\|^2_{H^1( Q^+_R )} + \| u \|^2_{ C^{\alpha}(\pa ' Q^+_R )} + \| f \|^2_{ C^{\alpha}(\pa ' Q^+_R )}\right).
\end{split}
\ee

\end{thm}

\begin{rem}\label{rem:localize problem general case}

For the loacl Schauder estimate Theorem \ref{thm:schauder estimate local general}, $ a_{ij}(x,t) $ does not necessarily have to be exactly the form \eqref{eq:flatten coefficient matrix}. Besides, our method of this Schauder type estimate for equation \eqref{eq:boundary diffusion local standard}  works for equation \eqref{eq:boundary diffusion local} as well, since the idea is of freezing coefficients.

\end{rem}

Finally, we look for the global estimate of solutions in $C^{1+\alpha}(\pa \Omega \times (0,T))$ to \eqref{eq:boundary diffusion} through the local estimate \eqref{eq:schauder estimate local general} with coefficients of \eqref{eq:bound of coefficients}. Thus we conclude that the following global Schauder estimates are natural consequence of Theorem\ref{thm:schauder estimate local general}.

\begin{defn}\label{defn:weak solution global}

We say $v$ is a weak solution  to \eqref{eq:boundary diffusion}, if 

\item[(i)] $v \ge 0$
\item[(ii)] $ v \in C([t_1,t_2]:L^1(\pa \om)) \bigcap L^2([t_1,t_2]:H^1(\om))$ 
\item[(iii)] 
\be \nonumber
\begin{split}
\int_{t_1}^{t_2} \int_{\pa \om} \left( v \pa_t \phi -bv\phi +f\phi  \right) \,\ud \sigma_x  \ud t  - \int_{t_1}^{t_2} \int_{ \om} a \nabla v \cdot \nabla \phi \,\ud x \ud t \\
= \int_{\pa \om } v(x,t_2)\phi(x,t_2) \, \ud \sigma_x - \int_{\pa \om } v(x,t_1)\phi(x,t_1) \, \ud \sigma_x
\end{split}
\ee
hold for any $0<t_1<t_2<T$ and $\phi(x,t) \in C^{\infty}(\overline{\om}\times [t_1,t_2])$. $\sigma_x$ denotes the element surface area for $x \in \pa \om$.

\end{defn}

\begin{defn}\label{defn:compatibility-global}

We say  a weak solution $v$ to \eqref{eq:boundary diffusion} is a CIE solution, if \eqref{eq:compatibility} holds in every locally flatten equation in $Q_R^+(x,t)$, for any $(x,t) \in \pa \om \times (0,T), R>0$.

\end{defn}

\begin{thm}\label{thm:schauder estimate global}
If $v \in C^{1+\alpha}(\om \times (0,T)) \bigcap C^{1+\alpha}(\pa \om \times (0,T))$ is a CIE solution to \eqref{eq:boundary diffusion} with the coefficients satisfying \eqref{eq:bound of coefficients} holding for every local equation in $Q_R^+(x,t), (x,t) \in \pa \om \times (0,T)$, then we have Schauder type estimate
\be\label{eq:schauder estimate global}
\begin{split}
&[v]^2_{C^{1+\alpha}(\pa \Omega \times (0,T) )}+[v]^2_{C^{1+\alpha}( \Omega \times (0,T) )} \\
&\le C(n,\Lambda,\alpha)\left( \|v\|^2_{H^1(\pa \Omega \times (0,T) )}+\|v\|^2_{H^1( \Omega \times (0,T)  )} + \| v \|^2_{ C^{\alpha}( \pa \Omega \times (0,T) )} + \| f \|^2_{ C^{\alpha}(\pa \Omega \times (0,T))}\right).
\end{split}
\ee

\end{thm}

The consequence of the above Schauder estimate is the existence of the CIE solution of the the global equation \eqref{eq:boundary diffusion} in
$C^{1+\alpha}(\om \times (0,T)) \bigcap C^{1+\alpha}(\pa \om \times (0,T))$, guided by the compatibility \eqref{eq:compatibility}. These applications are presented in Section \ref{sec:application}.

\section{Very weak solutions and weak solutions}\label{sec:weak solution}

In this section, we introduce some related definitions for the local equation \eqref{eq:boundary diffusion local standard} as preparation. The energy spaces similar to those that are used for parabolic equations gives the notation of very weak solution.

\begin{defn}\label{defn:veryweak solution of localized problem}
We say that $u\in L^2\left(I_R: H^1 (  B^+_R) \right) \bigcap L^{\infty}\left(I_R:L^2(\pa' B^+_R)\right) $ is a very weak solution to \eqref{eq:boundary diffusion local standard} if it satisfies
\be\label{eq:very weak solution of localized problem}
-\int_{\pa' Q^+_R } u \pa_t \phi   + \int_{Q^+_R}a \tilde{\varphi} a_{ij}\pa_iu\pa_j \phi   = \int_{\pa' Q^+_R }\left(f-bu  \right)\phi - \int_{Q^+_R}a a_{ij}\pa_iu\pa_j\tilde{\varphi}   \phi
\ee
for some $f\in L^2\left(  \pa' Q^+_R   \right)$ and every $ \phi \in C_c^{\infty}(Q_R)$. In addition, u is called an entire solution to 
\begin{equation*}
 \left\{
\begin{aligned}
  \pa_i(a(x,t)a_{ij}(x,t)\pa_j u)  &=   0  \quad  \mbox{in }    \R^n_+ \times \R,     \\
   \partial_{t}u-a(x,t) \tilde{\varphi}(x) a_{nj}(x,t)\pa_ju + bu &=  f \quad  \mbox{on }   \pa  \R^n_+ \times  \R
\end{aligned}
\right.
\end{equation*} 
if $u$ is a very weak solution to \eqref{eq:boundary diffusion local standard} for every $R>0$ and $\varphi$ is a constant function.
\end{defn}
For very weak solutions, one can employ DeGiorgi-Nash-Moser iteration to get $L^{\infty}$ estimate on $ \pa' Q^+_{R/2} $ as what has been done in \cite{2024JXY} and  further explore Harnack type inequality. But, for the Schauder type estimate, we have to work with weak solution, which is viewed as a `` Stronger'' version of the very weak solution, to get gradient estimate. 

\begin{defn}\label{defn:weak solution of localized problem}

$ u \in  H^{1}\left(\pa ' Q^+_R \right) \bigcap L^{\infty}\left(I_R:H^1(B_R^+)\right) $ is called a weak solution to \eqref{eq:boundary diffusion local standard} if it satisfies
\be\label{eq:weak solution of localized problem}
\int_{\pa' Q^+_R } \phi \pa_t u     - \int_{\pa' Q^+_R}a \tilde{\varphi} a_{nj}\pa_ju \phi  = \int_{\pa' Q^+_R }\left(f-bu  \right)\phi      
\ee
for some $f\in L^2\left(  \pa' Q^+_R   \right)$ and every $ \phi \in C_c^{\infty}(Q_R)$.

\end{defn}

The existence of the very weak solutions is guided by the natural energy estimate, as the existence of the CIE solutions is derived under the energy estimate and the compatibility \eqref{eq:compatibility}.

\begin{rem}\label{rem:steklov average}
In the study of weak solutions, sometimes we need to employ the Steklov average
$$
u_h(x,t):=\frac{1}{h}\int_t^{t+h}u(x,\tau)d\tau
$$
where $h>0$ and $u$ is a given function.

From \cite{1996Lie}, we know that if $u\in L^2\left(I_R: H^1 (  B^+_R) \right)$ and $\delta>0$, then
$$
u_h \to u, \quad \nabla u_h \to \nabla u \quad \mbox{in }L^2\left(  B^+_R \times (-R,R-\delta) \right) 
$$
as $h \to 0$. In addition, if $u$ is a very weak solution to \eqref{eq:boundary diffusion local standard}, then $u_h$ satisfies
\be\label{eq:steklov average of veryweak solution of localized problem}
\int_{\pa' Q^+_R }\phi  \pa_t u_h    + \int_{Q^+_R}\left(a \tilde{\varphi} a_{ij}\pa_iu\right)_h \pa_j \phi = \int_{\pa' Q^+_R }\left(f-bu- a a_{ij}\pa_iu\pa_j\tilde{\varphi}  \right)_h \phi 
\ee
for every $ \phi \in C_c^{\infty}(B_R \times (-R,R-h))$. More properties of Steklov averages can be found in \cite{2017CDG}.
\end{rem}
With regard to the interior estimate of the weak solution, we can use that $\pa_t u \in L^2(  \pa ' Q_{R}^+)$ and directly work with it, because equivalently we could work with the Steklov average $u_h$ of $u$ and later take the limit as $h \to 0$.

In order to get local Schauder estimates for \eqref{eq:boundary diffusion local standard}, we search estimates for corresponding equations with constant coefficients at the beginning. The related equations are as follow
\begin{equation}\label{eq:diffusion-local-constant coefficients-Non homogeous}
 \left\{
\begin{aligned}
  \pa_i(a_{ij}\pa_j u)  &=   \pa_iF_i \quad  \mbox{in }    Q^+_R= B^+_R \times I_R,     \\
   \partial_{t}u-  a_{nj}\pa_ju  &=  f \quad  \mbox{on }  \pa ' Q^+_R =  \pa ' B^+_R \times I_R.
\end{aligned}
\right.
\end{equation}
where  $a_{ij} $ are constant . And we shall also look into the homogeneous case of \eqref{eq:diffusion-local-constant coefficients-Non homogeous}
\begin{equation}\label{eq:diffusion-local-constant coefficients-homogeous}
 \left\{
\begin{aligned}
  \pa_i(a_{ij}\pa_j u)  &=   0  \quad  \mbox{in }    Q^+_R= B^+_R \times I_R,     \\
   \partial_{t}u-  a_{nj}\pa_ju  &= 0 \quad  \mbox{on }  \pa ' Q^+_R =  \pa ' B^+_R \times I_R.
\end{aligned}
\right.
\end{equation}
As previously done, we introduce the very weak and weak solutions to constant-coefficients equations \eqref{eq:diffusion-local-constant coefficients-Non homogeous} and \eqref{eq:diffusion-local-constant coefficients-homogeous} .
\begin{defn}\label{defn:veryweak solution of diffusion-local-constant coefficients-Non homogeous}
We say that $u\in L^2\left(I_R: H^1 (  B^+_R) \right) \bigcap L^{\infty}\left(I_R:L^2(\pa' B^+_R)\right) $ is a very weak solution to \eqref{eq:diffusion-local-constant coefficients-Non homogeous} if it satisfies
\be\label{eq:veryweak solution of diffusion-local-constant coefficients-Non homogeous}
-\int_{\pa' Q^+_R } u \pa_t \phi   + \int_{Q^+_R} a_{ij}\pa_iu\pa_j \phi = \int_{\pa' Q^+_R }\left(f+F_n\right)\phi + \int_{Q^+_R} F_i \pa_i\phi  
\ee
for some $f\in L^2\left(  \pa' Q^+_R   \right)$,  $F\in L^2\left(   Q^+_R   \right)^n$, $F_n \in L^2\left(  \pa' Q^+_R   \right)$ and every $ \phi \in C_c^{\infty}(Q_R)$. 
\end{defn}
For Schauder type estimates, we need to work on CIE solutions in $H^{1}\left(\pa ' Q^+_R \bigcup Q^+_R \right)$ in order to get the continuity information of $\pa_t u$ and $\nabla u $ on $\pa ' Q^+_R $. 
\begin{defn}\label{defn:weak solution of diffusion-local-constant coefficients-Non homogeous}
We say that $ u \in  H^{1}\left(\pa ' Q^+_R \right) \bigcap L^{\infty}\left(I_R:H^1(B_R^+)\right) $ is a weak solution to \eqref{eq:diffusion-local-constant coefficients-Non homogeous} if it satisfies
\be\label{eq:weak solution of diffusion-local-constant coefficients-Non homogeous}
\int_{\pa' Q^+_R }  \phi \pa_t u   - \int_{\pa'Q^+_R} a_{nj}\pa_ju \phi = \int_{\pa' Q^+_R }f\phi   
\ee
for some $f\in L^2\left(  \pa' Q^+_R   \right)$ and every $ \phi \in C_c^{\infty}(Q_R)$. 
\end{defn}
In homogeneous case, corresponding definition is written as
\begin{defn}\label{defn:weak solution of diffusion-local-constant coefficients-homogeous }
$ u \in  H^{1}\left(\pa ' Q^+_R \right) \bigcap L^{\infty}\left(I_R:H^1(B_R^+)\right) $ is called a weak solution to \eqref{eq:diffusion-local-constant coefficients-homogeous} if it satisfies
\be\label{eq:weak solution of diffusion-local-constant coefficients-homogeous}
\int_{\pa' Q^+_R }  \phi \pa_t u   - \int_{\pa'Q^+_R} a_{nj}\pa_ju \phi = 0 
\ee
for every $ \phi \in C_c^{\infty}(Q_R)$.
\end{defn}

 \section{Estimates under constant coefficients} \label{sec:constant coefficients}
 
In order to execute the Schauder type iteration, in this section, we first derive some  Caccioppoli type inequalities for the weak solutions of the equations with constant coefficients. Then we promote the regularity for CIE solutions, using these auxiliary estimates.

For the center $(0,0) \in \pa ' Q^+_{R}$, we denote
$$
\pa ' Q^+_\rho:= \pa ' Q^+_\rho(0,0)=\pa ' B^+_\rho \times (-\rho,\rho),
$$
and
$$
Q^+_\rho:= Q^+_\rho(0,0)=B^+_\rho \times (-\rho,\rho),
$$
for $\rho \in (0,R]$. We first investigate the homogeneous case \eqref{eq:diffusion-local-constant coefficients-homogeous}. By putting different test functions into \eqref{eq:weak solution of diffusion-local-constant coefficients-homogeous}, we get the desired Caccioppoli type inequalities. The method similar to the  parabolic equations provides us the estimates on
 $$
 \int_{ Q^+_\rho  } |D_x u|^2.
 $$
\begin{lem}\label{lem:homogeneous-caccioppoli-1}
If $u\in L^2\left(I_R: H^1 (  B^+_R) \right) \bigcap L^{\infty}\left(I_R:L^2(\pa' B^+_R)\right) $ is a weak solution to \eqref{eq:diffusion-local-constant coefficients-homogeous}, then there holds the following Caccioppoli inequality 
\begin{equation}\label{eq:homogeneous-caccioppoli-1}
  \begin{split}
      \sup_{I_\rho}  \int_{\pa ' B^+_\rho  } u^2          +    \int_{ Q^+_\rho  } |D_x u|^2    
                   \leq  C(n,\Lambda) \left(  \frac{1}{|R-\rho|^2} \int_{ Q^+_R  }  u^2  + \frac{1}{R-\rho}  \int_{ \pa' Q^+_R  }  u^2   \right)
  \end{split}
\end{equation}
for any $\rho \in (0,R]$.
\end{lem}

\begin{proof}
Let $\xi(t)\in C^\infty_c(I_R)$ and $\eta(x)\in C^\infty_c(B_R)$ be cut-off functions that
\begin{equation*}
\begin{aligned}
    \xi(t)=1 \mbox{ in } I_\rho , \quad \eta(x)=1  \mbox{ in } B_\rho, \\
   |D_t \xi| + | D_x \eta | \le \frac{C}{R-\rho}.
\end{aligned}
\end{equation*}
Multiplying the equation \eqref{eq:diffusion-local-constant coefficients-homogeous} with $\xi^2 \eta^2 u $ and integrating over $_\tau Q^+_R=B^+_R\times (-R,t)$ with $\tau \in I_\rho$, we get the above estimate \eqref{eq:homogeneous-caccioppoli-1} through integration by parts and Cauchy's inequality. 
\end{proof}

Next, we need to derive estimates on
$$
\int_{ \pa ' Q^+_\rho  } |\pa_t u|^2+ \int_{ \pa ' Q^+_\rho  } |D_x u|^2 .
$$
Because we lose the global estimate
$$
\int_{ \pa \om  } |D_x u|^2 \le C \int_{\pa \om} \left|\frac{\pa}{\pa_\nu} u\right|^2, \mbox{ if  } \Delta u =0  \mbox{  in  } \om,
 $$
two additional test functions are used to get additional estimates on $\int_{\pa ' Q^+_\rho  } |D_x u|^2$.
\begin{lem}\label{lem:homogeneous-caccioppoli-1.5}
If $ u \in  H^{1}\left(\pa ' Q^+_R \right) \bigcap L^{\infty}\left(I_R:H^1(B_R^+)\right) $ is a weak solution to \eqref{eq:diffusion-local-constant coefficients-homogeous}, then there holds the following Caccioppoli inequality 
\begin{equation}\label{eq:homogeneous-caccioppoli-1.5}
  \begin{split}
    \sup_{I_\rho}\int_{ B^+_\rho }  |D_x u|^2   +      \int_{\pa ' Q^+_\rho  } |D_x u|^2    
                   \leq   \frac{C(n,\Lambda)}{R-\rho} \left(    \int_{ Q^+_R  }  |D_x u|^2  +  \int_{  Q^+_R  }  | u_t|^2   \right)
  \end{split}
\end{equation}
for any $\rho \in (0,R]$.
\end{lem}

\begin{proof}
Let $\xi(t)\in C^\infty_c(I_R)$ and $\eta(x)\in C^\infty_c(B_R)$ be cut-off functions such that
\begin{equation*}
\begin{aligned}
    \xi(t)=1 \mbox{ in } I_\rho , \quad \eta(x)=1  \mbox{ in } B_\rho, \\
   |D_t \xi| + | D_x \eta | \le \frac{C}{R-\rho}.
\end{aligned}
\end{equation*}
We multiply the equation with $\xi^2 \eta^2 a_{nk} \pa_k u $ and integrate over $_\tau Q^+_R=B^+_R\times (-R,t)$ with $\tau \in I_\rho$. Using Definition \ref{defn:weak solution of diffusion-local-constant coefficients-homogeous } and divergence theorem , we see
$$
\int_{ _\tau \pa'Q^+_R }   u_t \left( \xi^2 \eta^2 a_{nk} \pa_k u \right)  + \int_{ _\tau Q^+_R  } a_{ij}  \pa_i u \pa_j \left(  \xi^2 \eta^2 a_{nk} \pa_k u     \right)  
 =0.
$$
Simple computation gives
\begin{equation*}
  \begin{split}
& \int_{ _\tau \pa' Q^+_R }  \xi^2 \eta^2   u_t \left( a_{nk} \pa_k u  \right)  + \frac{ a_{nk}}{2} \int_{ _\tau Q^+_R }\pa_k \left(   \xi^2 \eta^2  a_{ij}\pa_i u\pa_j u       \right)   \\
& =a_{nk} \int_{ _\tau Q^+_R } \xi^2 \eta (\pa_k \eta) a_{ij}\pa_i u \pa_j u  - 2a_{nk} \int_{ _\tau Q^+_R } (\pa_k u) \eta \xi^2  a_{ij}\pa_i u \pa_j\eta.
  \end{split}
\end{equation*}
Now, we use divergence theorem and the fact that $\nu=(\nu_1,\nu_2,...,\nu_n)=(0,0,...,-1)$ on $\pa' B_R^+$, showing
\begin{equation*}
  \begin{split}
& \int_{ _\tau Q^+_R }  \pa_l\left( \xi^2 \eta^2 u_t a_{lk}\pa_k u         \right)+ \frac{a_{nn}}{2}\int_{ _\tau \pa' Q^+_R }  \xi^2 \eta^2  a_{ij}\pa_i u\pa_j u          \\
& =-a_{nk} \int_{ _\tau Q^+_R } \xi^2 \eta (\pa_k \eta) a_{ij}\pa_i u \pa_j u  + 2a_{nk} \int_{ _\tau Q^+_R } (\pa_k u) \eta \xi^2  a_{ij}\pa_i u \pa_j\eta.
  \end{split}
\end{equation*}
Therefore,
\begin{equation*}
  \begin{split}
& \frac{1}{2}\int_{ _\tau Q^+_R }\frac{d}{dt}\left( \xi^2 \eta^2 a_{lk}\pa_l u \pa_k u\right) + \frac{a_{nn}}{2}\int_{ _tau \pa' Q^+_R }  \xi^2 \eta^2  a_{ij}\pa_i u\pa_j u          \\
& =-a_{nk} \int_{ _\tau Q^+_R } \xi^2 \eta (\pa_k \eta) a_{ij}\pa_i u \pa_j u  + 2a_{nk} \int_{ _\tau Q^+_R } (\pa_k u) \eta \xi^2  a_{ij}\pa_i u \pa_j\eta \\
& -2\int_{ _\tau Q^+_R }\eta \xi^2 u_t a_{lk}\pa_l u \pa_k \eta +\int_{ _\tau Q^+_R } \xi \xi_t \eta^2 a_{lk}\pa_l u \pa_k u  .
  \end{split}
\end{equation*}
Utilizing Cauchy's inequality and the fact that 
$$
\frac{1}{\Lambda} \le a_{nn} \le \Lambda,
$$  
we get
  \begin{equation*}
  \begin{split}
& \int_{ B^+_R \times {\tau} } \left( \xi^2 \eta^2  a_{ij}\pa_i u \pa_j u  \right)   +\int_{ _\tau \pa' Q^+_R }  \xi^2 \eta^2  a_{ij}\pa_i u\pa_j u  \\
& \le C\left(  \int_{ _\tau Q^+_R } \xi^2  \eta |D_x \eta|   |u_t|^2    +   \int_{ _\tau Q^+_R }  \xi^2 \eta |D_x \eta| |D_x u|^2 + \int_{ _\tau Q^+_R } \xi^2 \eta |\nabla \eta| a_{ij}\pa_i u \pa_j u + \int_{ _\tau Q^+_R } \xi \eta^2 |\xi_t|  a_{ij}\pa_i u \pa_j u \right).
  \end{split}
\end{equation*}
The lemma is proved using that $\tau \in I_\rho$ is arbitrary, and   combining above inequality with the ellipticity of $a_{ij}$ and the  estimates on the cut-off functions.

\end{proof}

\begin{lem}\label{lem:homogeneous-caccioppoli-2}
If $ u \in  H^{1}\left(\pa ' Q^+_R \right) \bigcap L^{\infty}\left(I_R:H^1(B_R^+)\right) $ is a weak solution to \eqref{eq:diffusion-local-constant coefficients-homogeous}, then there holds the following Caccioppoli inequality 
\begin{equation} \label{eq:homogeneous-caccioppoli-2}
  \begin{split}
& \sup_{I_\rho}\int_{ B^+_\rho }  |D_x u|^2   + \int_{  \pa 'Q^+_\rho  }  |u_t|^2  \le    \frac{C(n,\Lambda)}{R-\rho} \left(    \int_{ Q^+_R  }  |D_x u|^2  +  \int_{  Q^+_R  }  | u_t|^2   \right) .
 \end{split}
\end{equation}
for any $\rho \in (0,7R/8]$.
\end{lem}

\begin{proof}

We put $\xi^2 \eta^2 u_t $ into \eqref{eq:weak solution of diffusion-local-constant coefficients-homogeous} and integrate over $_tQ^+_R$. Then integration by parts gives
\begin{equation*}
  \begin{split}
& \frac{1}{2}\int_{ _tQ^+_R }\frac{d}{dt}\left( \xi^2 \eta^2 a_{ij}\pa_i u \pa_j u \right) + \int_{ _t \pa 'Q^+_R  }  \xi^2 \eta^2 |u_t|^2 \\
& =\int_{ _tQ^+_R } \xi \xi_t \eta^2 a_{ij}\pa_i u \pa_j u  -  \int_{ _tQ^+_R } 2 \eta \xi^2 u_t a_{ij}\pa_i u \pa_j\eta.
  \end{split}
\end{equation*}
For the second term on the right hand side, we have
\begin{equation*}
  \begin{split}
&  -  \int_{ _\tau Q^+_R } 2 \eta \xi^2 u_t a_{ij}\pa_i u \pa_j\eta \\
& \le C  \int_{ _\tau Q^+_R } \xi^2  \eta  |D_x \eta|  |u_t|^2    + C \int_{ _\tau Q^+_R }  \xi^2  \eta  |D_x  \eta| |D_x  u|^2   .
\end{split}
\end{equation*}
Combining above two estimates, we obtain
\begin{equation*}
  \begin{split}
& \int_{ _\tau Q^+_R }\frac{d}{dt} \left( \xi^2 \eta^2 a_{ij}\pa_i u \pa_j u  \right)+ \int_{ _\tau \pa 'Q^+_R  }  \xi^2 \eta^2 |u_t|^2 \\
& \le C\int_{ _\tau Q^+_R } \xi |\xi_t| \eta^2 |D_x  u|^2 +  C  \int_{ _\tau Q^+_R } \xi^2  \eta  |D_x \eta|  |u_t|^2    + C \int_{ _\tau Q^+_R }  \xi^2  \eta  |D_x  \eta| |D_x  u|^2     .
  \end{split}
\end{equation*}
Since $\tau \in I_\rho$ is arbitrary, we use estimates on cut-off functions to find
 \begin{equation*}
  \begin{split}
& \sup_{I_\rho}\int_{ B^+_\rho }  |D_x u|^2   + \int_{  \pa 'Q^+_\rho  }  |u_t|^2  \le   C(n,\Lambda) \left(         \frac{1}{R-\rho}  \int_{  Q^+_R  } |D_x  u|^2   +  \frac{1}{R-\rho}  \int_{  Q^+_R  } |u_t|^2 \right)   .
  \end{split}
\end{equation*}

\end{proof}

Now, we derive some regularity estimates for CIE solutions of homogeneous equation \eqref{eq:diffusion-local-constant coefficients-homogeous} from the compatibility \eqref{eq:compatibility} and the Cacciopolli inequalities in above three Lemmas.
\begin{lem}\label{lem:homogeneous-iteration}
If $ u \in  H^{1}\left(\pa ' Q^+_R \right) \bigcap L^{\infty}\left(I_R:H^1(B_R^+)\right) $ is a CIE solution to \eqref{eq:diffusion-local-constant coefficients-homogeous}, then $u\in C^{\infty}_{loc}\left(\pa ' Q^+_{\frac{3}{4}R} \right)\bigcap C^{\infty}_{loc}\left( Q^+_{\frac{3}{4}R}\right) $ and for any $\rho \in (0,R/2]$, there holds 
\begin{equation} \label{eq:homogeneous-iteration-1}
  \begin{split}
          \fint_{\pa ' Q^+_\rho  }  u^2 + \fint_{Q^+_\rho  }  u^2  & \le C(n,\Lambda) \left(\fint_{\pa ' Q^+_R  }  u^2 +\fint_{ Q^+_R  }  u^2 \right)
\end{split}
\end{equation}
and
 \begin{equation} \label{eq:homogeneous-iteration-2}
    \begin{split}
     &\fint_{\pa ' Q^+_\rho  }|u(x,t)-u(y,s)|^2dx'dt +  \fint_{ Q^+_\rho  }|u(x,t)-u(y,s)|^2dxdt      \\
      & \le C(n,\Lambda)\left(\frac{\rho}{R}\right)^2   \left( \fint_{\pa ' Q^+_R  } |u(x,t)-\lambda_1|^2dx'dt+\fint_{ Q^+_R  } |u(x,t)-\lambda_1|^2dxdt\right)
     \end{split}   
\end{equation}
and
 \begin{equation} \label{eq:homogeneous-iteration-3}
    \begin{split}
     &\fint_{\pa ' Q^+_\rho  }|u(x,t)-u(z,\tau)|^2dx'dt +  \fint_{ Q^+_\rho  }|u(x,t)-u(z,\tau)|^2dxdt      \\
      & \le C(n,\Lambda)\left(\frac{\rho}{R}\right)^2   \left( \fint_{\pa ' Q^+_R  } |u(x,t)-\lambda_2|^2dx'dt+\fint_{ Q^+_R  } |u(x,t)-\lambda_2|^2dxdt\right)
     \end{split}   
\end{equation}
where $(y,s) \in \pa ' Q^+_\rho, (z,\tau) \in  Q^+_\rho$ and  $\lambda_1, \lambda_2$ are two arbitrary real numbers.
\end{lem}
\begin{proof}
By the scaling 
$$
v(x,t)=u(\frac{x}{R},\frac{t}{R}),
$$
it is sufficient to prove Lemma \ref{lem:homogeneous-iteration} for $R=1$. Combining Lemama \ref{lem:homogeneous-caccioppoli-1}, Lemma \ref{lem:homogeneous-caccioppoli-1.5}, Lemma \ref{lem:homogeneous-caccioppoli-2} and \eqref{eq:compatibility}, 
$$
\int_{\pa ' Q^+_\rho  } |D_{x} u|^2+\int_{\pa ' Q^+_\rho  } |\pa_t u|^2 \le C\left( \frac{1}{(1-\rho)^2} \int_{\pa ' Q^+_1  } u^2 +\frac{1}{(1-\rho)^3} \int_{  Q^+_1  } u^2 \right) ,\quad \rho\in (0,3/4).
$$
Consequently, one produces
$$
\| u \|^2_{H^{1}(\pa ' Q^+_\rho)}+\| u \|^2_{H^{1}( Q^+_\rho)}\le  C\left( \frac{1}{(1-\rho)^2} \int_{\pa ' Q^+_1  } u^2 +\frac{1}{(1-\rho)^3} \int_{ Q^+_1  } u^2 \right) ,\quad \rho\in (0,3/4).
$$
We use difference quotients repeatedly to above inequality, obtaining $u\in H^{m}_{loc} \left(\pa ' Q^+_{\frac{3}{4}R}  \right)$ and $ u\in H^{m}_{loc} \left(Q^+_{\frac{3}{4}R}\right) $ for any integer $m$. This proves the first assertion. Repeatedly utilizing above procedures,  
$$
\| u \|^2_{H^{k}(\pa ' Q^+_{1/2})}+\| u \|^2_{H^{k}( Q^+_{1/2})}\le  C(k) \left(\int_{\pa ' Q^+_1  } u^2 +\int_{ Q^+_1  } u^2\right)
$$
for some large $k$. Then, by Sobolev embedding theorem, we know
$$
\| u \|^2_{C^{1}(\pa ' Q^+_{1/2}) }+\| u \|^2_{C^{1}( Q^+_{1/2}) }\le  C \left(\int_{\pa ' Q^+_1  } u^2 +\int_{ Q^+_1  } u^2\right). 
$$ 
Thus
\begin{equation*}
\begin{split}
\fint_{\pa ' Q^+_\rho  }  u^2  \le C\sup_{\pa ' Q^+_{1/2} } u^2 &\le C \left( \int_{\pa ' Q^+_1  } u^2+\int_{ Q^+_1  } u^2\right)  \\
&\le C \left( \fint_{\pa ' Q^+_1  } u^2+\fint_{ Q^+_1  } u^2\right), \quad \rho \in (0,1/2]
\end{split}
\end{equation*}
and
\begin{equation*}
\begin{split}
\fint_{ Q^+_\rho  }  u^2  \le C\sup_{ Q^+_{1/2} } u^2 &\le C \left( \int_{\pa ' Q^+_1  } u^2+\int_{ Q^+_1  } u^2\right)  \\
&\le C \left( \fint_{\pa ' Q^+_1  } u^2+\fint_{ Q^+_1  } u^2\right), \quad \rho \in (0,1/2].
\end{split}
\end{equation*}
\eqref{eq:homogeneous-iteration-1} is proved. Using the continuity,
\begin{equation*} 
  \begin{split}
&  \fint_{\pa ' Q^+_\rho  }|u(x,t)-u(y,s)|^2dx'dt  \le  ( \  \mathop{osc}\limits_{\pa ' Q^+_\rho}u  \ )^2 \\
& \le C\rho^2 \| u \|^2_{C^{1}(\pa ' Q^+_{1/2})} \\
& \le C\rho^2 \left(  \int_{\pa ' Q^+_1  } u^2+\int_{ Q^+_1  } u^2 \right), \quad \rho \in (0,1/2]  .
 \end{split}
\end{equation*}
Similarly,
\begin{equation*} 
  \begin{split}
&  \fint_{ Q^+_\rho  }|u(x,t)-u(y,s)|^2dxdt  \le  ( \  \mathop{osc}\limits_{ Q^+_\rho}u  \ )^2 \\
& \le C\rho^2 \| u \|^2_{C^{1}( Q^+_{1/2})} \\
& \le C\rho^2 \left(  \int_{\pa ' Q^+_1  } u^2+\int_{ Q^+_1  } u^2 \right), \quad \rho \in (0,1/2]  .
 \end{split}
\end{equation*}
\eqref{eq:homogeneous-iteration-2} is proved if we put $u-\lambda_1$ into above two estimates. The proof of \eqref{eq:homogeneous-iteration-3} is similar.  
\end{proof}
From Lemma \ref{lem:homogeneous-iteration}, one directly concludes the following result.
\begin{cor}\label{cor:homogeneous-iteration}
If $ u \in  H^{1}\left(\pa ' Q^+_R \right) \bigcap L^{\infty}\left(I_R:H^1(B_R^+)\right) $ is a CIE solution to \eqref{eq:diffusion-local-constant coefficients-homogeous}, then $u\in C^{\infty}_{loc}\left(\pa ' Q^+_{\frac{3}{4}R} \right)\bigcap C^{\infty}_{loc}\left( Q^+_{\frac{3}{4}R}\right) $ and for any $\rho \in (0,R/2]$, there holds 
\begin{equation} \label{eq:homogeneous-iteration-average-1}
   \begin{split}
          \fint_{\pa ' Q^+_\rho  }  u^2 + \fint_{Q^+_\rho  }  u^2  & \le C(n,\Lambda) \left(\fint_{\pa ' Q^+_R  }  u^2 +\fint_{ Q^+_R  }  u^2 \right)
\end{split}
\end{equation}
and
\begin{equation} \label{eq:homogeneous-iteration-average-2a}
  \begin{split}
     &\fint_{\pa ' Q^+_\rho  }|u(x,t)-u_\rho|^2dx'dt+\fint_{ Q^+_\rho  }|u(x,t)-u_\rho|^2dxdt \\
     &+\fint_{\pa ' Q^+_\rho  }|u(x,t)-\tilde{u_\rho}|^2dx'dt+\fint_{ Q^+_\rho  }|u(x,t)-\tilde{u_\rho}|^2dxdt\\
      & \le C(n,\Lambda)\left(\frac{\rho}{R}\right)^2   \Bigg( \fint_{\pa ' Q^+_R  } |u(x,t)-u_R|^2dx'dt+\fint_{ Q^+_R  } |u(x,t)-u_R|^2dxdt\\
     &+\fint_{\pa ' Q^+_R  } |u(x,t)-\tilde{u_R}|^2dx'dt+\fint_{ Q^+_R  } |u(x,t)-\tilde{u_R}|^2dxdt \Bigg)
     \end{split} 
\end{equation}
where 
$$
u_R=\fint_{\pa ' Q^+_R  }   u   \quad \mbox{and}\quad \tilde{u_R}=\fint_{ Q^+_R  } u.
$$
\end{cor}

Now, we look into the non-homogeneous case \eqref{eq:diffusion-local-constant coefficients-Non homogeous}. The semi norms corresponding to equation \eqref{eq:diffusion-local-constant coefficients-Non homogeous} are

$$
[u]^2_{H^1(\pa ' Q^+_R )}=  \| D_{x} u \|^2_{L^2(\pa ' Q^+_R)} + \| \pa_{t} u \|^2_{L^2(\pa ' Q^+_R)}= \sum_{j=1}^{n}\| \pa_{j} u \|^2_{L^2(\pa ' Q^+_R)} + \| \pa_{t} u \|^2_{L^2(\pa ' Q^+_R)},
$$
$$
[u]^2_{H^1( Q^+_R )}=  \| D_{x} u \|^2_{L^2( Q^+_R)} + \| \pa_{t} u \|^2_{L^2( Q^+_R)}= \sum_{j=1}^{n}\| \pa_{j} u \|^2_{L^2( Q^+_R)} + \| \pa_{t} u \|^2_{L^2( Q^+_R)},
$$
$$
(f)_R=(f)_{\pa ' Q^+_R} = \fint_{\pa ' Q^+_R  }f .
$$
And we define the polynomials
\begin{subequations}\label{eq:polynomial on the boundary}
\begin{align}
&P(u,R)=  \sum_{j=1}^{n} (\pa_j u)_{R} x_j + (\pa_t u)_{R} t ,\\
&\tilde{P}(u,R)=  \sum_{j=1}^{n} \tilde{(\pa_j u)_{R}} x_j + \tilde{(\pa_t u)_{R}} t,
\end{align}
\end{subequations}
for which the polynomial $P(u,R)$ satisfies
\begin{equation*}
 \left\{
\begin{aligned}
  \pa_i(a_{ij}\pa_j P)  &=   0 \quad  \mbox{in }    Q^+_R,     \\
   \partial_{t}P-  a_{nj}\pa_jP  &=  (f)_R \quad  \mbox{on }  \pa ' Q^+_R ,
\end{aligned}
\right.
\end{equation*}
if $u$ is a solution to \eqref{eq:diffusion-local-constant coefficients-Non homogeous}. By comparing the polynomials \eqref{eq:polynomial on the boundary} with the solution to \eqref{eq:diffusion-local-constant coefficients-Non homogeous}, we have the following result.
\begin{lem}\label{lem:non-homogeneous-iteration}
Suppose $ u \in  H^{1}\left(\pa ' Q^+_R \right) \bigcap L^{\infty}\left(I_R:H^1(B_R^+)\right) $ is a CIE solution to \eqref{eq:diffusion-local-constant coefficients-Non homogeous}, then for any $\rho \in (0,R]$ there holds 
\begin{equation} \label{eq:non-homogeneous-iteration-1}
  \begin{split}
                  & [u]^2_{H^1(\pa ' Q^+_\rho )}+ \frac{1}{\rho} [u]^2_{H^1( Q^+_\rho )}\\
     & \le C  \left(\frac{\rho}{R}\right)^{n} \left(  [u]^2_{H^1(\pa ' Q^+_R )} + \frac{1}{R}[u]^2_{H^1(Q^+_R )}\right)+  \frac{C}{R}\sum_{i=1}^n \| F_i \|^2_{L^2( Q^+_R)}\\
   &+C \| F_n\|^2_{L^2( \pa' Q^+_R)} + C\| f \|^2_{L^2(\pa ' Q^+_R)}  ,
 \end{split}
\end{equation}
\begin{equation} \label{eq:non-homogeneous-iteration-2}
  \begin{split}
         &[u-P(u,\rho)]^2_{H^1(\pa ' Q^+_\rho )}+\frac{1}{\rho}[u-P(u,\rho)]^2_{H^1(Q^+_\rho )}\\
         +&[u-\tilde{P}(u,\rho)]^2_{H^1(\pa ' Q^+_\rho )}+\frac{1}{\rho}[u-\tilde{P}(u,\rho)]^2_{H^1(Q^+_\rho )}\\
& \le  C\left(\frac{\rho}{R}\right)^{n+2} \Bigg( [u-P(u,R)]^2_{H^1(\pa ' Q^+_R )}+ \frac{1}{R} [u-P(u,R)]^2_{H^1(Q^+_R )}\\
&+[u-\tilde{P}(u,R)]^2_{H^1(\pa ' Q^+_R )}+ \frac{1}{R} [u-\tilde{P}(u,R)]^2_{H^1(Q^+_R )}\Bigg)\\
  &+  \frac{C}{R} \sum_{i=1}^n \| F_i \|^2_{L^2( Q^+_R)} +C\| F_n \|^2_{L^2(\pa ' Q^+_R)} + C\| f-(f)_R \|^2_{L^2(\pa ' Q^+_R)}.
 \end{split}
\end{equation}

\end{lem}

\begin{proof}
It suffices to merely prove Lemma \ref{lem:non-homogeneous-iteration} for $\rho \in (0,R/2]$, otherwise one can take $C\geq 2^{n+2}$.  In order to study the solution of \eqref{eq:diffusion-local-constant coefficients-Non homogeous}, it is convenient to investigate the following two initial value problems
\begin{equation}\label{eq:diffusion-local-constant coefficients-Non homogeous-initial-1}
 \left\{
\begin{aligned}
  \pa_i(a_{ij}\pa_j w_1)  &=   \pa_iF_i \quad  \mbox{in }    Q^+_R,     \\
   \partial_{t}w_1-  a_{nj}\pa_jw_1  &=  0 \quad   \quad \mbox{on }  \pa ' Q^+_R ,\\
   w_1\big|_{t=-R}&=0
\end{aligned}
\right.
\end{equation}
and
\begin{equation}\label{eq:diffusion-local-constant coefficients-Non homogeous-initial-2}
 \left\{
\begin{aligned}
  \pa_i(a_{ij}\pa_j w_2)  &=   0  \quad  \mbox{in }    Q^+_R,     \\
   \partial_{t}w_2-  a_{nj}\pa_jw_2  &=  f \quad  \mbox{on }  \pa ' Q^+_R ,\\
   w_2\big|_{t=-R}&=0.
\end{aligned}
\right.
\end{equation}
From the theory of linear parabolic equations and \eqref{eq:compatibility}, it is clear  that there exist CIE solutions  
$$
w_1, w_2 \in  H^{1}\left(\pa ' Q^+_R \right) \bigcap L^{\infty}\left(I_R:H^1(B_R^+)\right)
$$
of \eqref{eq:diffusion-local-constant coefficients-Non homogeous-initial-1} and \eqref{eq:diffusion-local-constant coefficients-Non homogeous-initial-2} respectively, with interior estimates
\be\label{eq:w1estimate}
[w_1]^2_{H^1(\pa ' Q^+_{3R/4} )}+ \frac{1}{R}[w_1]^2_{H^1( Q^+_{3R/4} )}\le  C \left( \frac{1}{R}\sum_{i=1}^n \| F_i \|^2_{L^2( Q^+_R)} + \| F_n \|^2_{L^2( \pa ' Q^+_R)}  \right) ,
\ee
\be\label{eq:w2estimate}
 [w_2]^2_{H^1(\pa ' Q^+_{3R/4} )} +\frac{1}{R}[w_2]^2_{H^1( Q^+_{3R/4} )} \le  C  \| f \|^2_{L^2(\pa ' Q^+_R)}.
\ee
Setting
$$
v:=u-w_1-w_2,
$$
and noting that if $u \in H^{1}\left(\pa ' Q^+_R \right) \bigcap L^{\infty}\left(I_R:H^1(B_R^+)\right)$ is a weak solution to \eqref{eq:diffusion-local-constant coefficients-Non homogeous}, then $v \in  H^{1}\left(\pa ' Q^+_R \right) \bigcap L^{\infty}\left(I_R:H^1(B_R^+)\right) $, and it is a CIE solution to the homogeneous equation \eqref{eq:diffusion-local-constant coefficients-homogeous}. By Lemma \ref{lem:homogeneous-iteration}, $v\in  C^{\infty}_{loc}\left(\pa ' Q^+_{\frac{3}{4}R} \right)\bigcap C^{\infty}_{loc}\left( Q^+_{\frac{3}{4}R}\right) $.
Thus applying  Corollary \ref{cor:homogeneous-iteration} to $\nabla v$ and $\pa_t v$, we obtain
\begin{equation} \label{eq:homogeneous iteration derivative 1}
  \begin{split}
        &  [v]^2_{H^1(\pa ' Q^+_\rho )}+ \frac{1}{\rho}[v]^2_{H^1( Q^+_\rho )}    \le C  \left(\frac{\rho}{R}\right)^{n} \left(  [v]^2_{H^1(\pa ' Q^+_R )} + \frac{1}{R}[v]^2_{H^1(Q^+_R )}\right),
     \end{split}
     \end{equation}
\begin{equation} \label{eq:homogeneous iteration derivative 2}
  \begin{split}
& [v-P(v,\rho)]^2_{H^1(\pa ' Q^+_\rho )}+ \frac{1}{\rho} [v-P(v,\rho)]^2_{H^1( Q^+_\rho )} \\
&+ [v-\tilde{P}(v,\rho)]^2_{H^1(\pa ' Q^+_\rho )}+ \frac{1}{\rho} [v-\tilde{P}(v,\rho)]^2_{H^1( Q^+_\rho )}\\
&  \le C  \left(\frac{\rho}{R}\right)^{n+2} \Bigg( [v-P(v,R)]^2_{H^1(\pa ' Q^+_R )}+ \frac{1}{R} [v-P(v,R)]^2_{H^1(Q^+_R )}\\
  & +   [v-\tilde{P}(v,R)]^2_{H^1(\pa ' Q^+_R )}+ \frac{1}{R} [v-\tilde{P}(v,R)]^2_{H^1(Q^+_R )}     \Bigg),
 \end{split}
     \end{equation}
for any $\rho \in (0,R/2]$ and $C$ only depends on $n,\Lambda$.  

Let
$$
u=v+w_1+w_2,
$$ 
then it follows from \eqref{eq:w1estimate}-\eqref{eq:homogeneous iteration derivative 1} that 
\begin{equation}\label{eq:nonhomogeneous iteration computation1}
  \begin{split}
          & [u]^2_{H^1(\pa ' Q^+_\rho )}+  \frac{1}{\rho}[u]^2_{H^1( Q^+_\rho )}\\
         &  \le 4\Big( [v]^2_{H^1(\pa ' Q^+_\rho )} +\frac{1}{\rho}[v]^2_{H^1( Q^+_\rho )}+ [w_1]^2_{H^1(\pa ' Q^+_\rho )}+\frac{1}{\rho}[w_1]^2_{H^1(Q^+_\rho )}\\
         &+ [w_2]^2_{H^1(\pa ' Q^+_\rho )}+\frac{1}{\rho}[w_2]^2_{H^1( Q^+_\rho )} \Big)
     \\
  & \le C  \left(\frac{\rho}{R}\right)^{n} \left(  [v]^2_{H^1(\pa ' Q^+_R )} + \frac{1}{R}[v]^2_{H^1(Q^+_R )}\right)+ \frac{C}{R}\sum_{i=1}^n \| F_i \|^2_{L^2( Q^+_R)}\\
   &+C \| F_n\|^2_{L^2( \pa' Q^+_R)} + C\| f \|^2_{L^2(\pa ' Q^+_R)}\\
    & \le C  \left(\frac{\rho}{R}\right)^{n} \left(  [u]^2_{H^1(\pa ' Q^+_R )} + \frac{1}{R}[u]^2_{H^1(Q^+_R )}\right)+ \frac{C}{R}\sum_{i=1}^n \| F_i \|^2_{L^2( Q^+_R)}\\
   &+C \| F_n\|^2_{L^2( \pa' Q^+_R)} + C\| f \|^2_{L^2(\pa ' Q^+_R)}.
      \end{split}
\end{equation} This finishes the proof of \eqref{eq:non-homogeneous-iteration-1}.

By \eqref{eq:w1estimate}-\eqref{eq:w2estimate} and \eqref{eq:homogeneous iteration derivative 2}, we compute
     \begin{equation}\label{eq:nonhomogeneous iteration computation2}
  \begin{split}
      &[u-P(u,\rho)]^2_{H^1(\pa ' Q^+_\rho )}+\frac{1}{\rho}[u    -P(u,\rho)]^2_{H^1(Q^+_\rho )}\\
      &+[u-\tilde{P}(u,\rho)]^2_{H^1(\pa ' Q^+_\rho )}+\frac{1}{\rho}[u    -\tilde{P}(u,\rho)]^2_{H^1(Q^+_\rho )}\\
            & \le 16\Big( [v-P(v,\rho)]^2_{H^1(\pa ' Q^+_\rho )} +\frac{1}{\rho}[v    -P(v,\rho)]^2_{H^1(Q^+_\rho )}\\
            &+[v-\tilde{P}(v,\rho)]^2_{H^1(\pa ' Q^+_\rho )}+\frac{1}{\rho}[v    -\tilde{P}(v,\rho)]^2_{H^1(Q^+_\rho )}\\
            &+ [w_1]^2_{H^1(\pa ' Q^+_\rho )} + \frac{1}{\rho}[w_1]^2_{H^1( Q^+_\rho )} + [w_2]^2_{H^1(\pa ' Q^+_\rho )}+ \frac{1}{\rho}[w_2]^2_{H^1( Q^+_\rho )} \Big)
     \\
  & \le C \left(\frac{\rho}{R}\right)^{n+2} \Bigg( [v-P(v,R)]^2_{H^1(\pa ' Q^+_R )}+ \frac{1}{R} [v-P(v,R)]^2_{H^1(Q^+_R )}\\
  &+     [v-\tilde{P}(v,R)]^2_{H^1(\pa ' Q^+_R )}+ \frac{1}{R} [v-\tilde{P}(v,R)]^2_{H^1(Q^+_R )}  \Bigg)\\
  &      +16\Big( [w_1]^2_{H^1(\pa ' Q^+_\rho )} + \frac{1}{\rho}[w_1]^2_{H^1( Q^+_\rho )} + [w_2]^2_{H^1(\pa ' Q^+_\rho )}+ \frac{1}{\rho}[w_2]^2_{H^1( Q^+_\rho )} \Big)      \\
  & \le C\left(\frac{\rho}{R}\right)^{n+2} \Bigg( [u-P(u,R)]^2_{H^1(\pa ' Q^+_R )}+ \frac{1}{R} [u-P(u,R)]^2_{H^1(Q^+_R )}  \\
 &+ [u-\tilde{P}(u,R)]^2_{H^1(\pa ' Q^+_R )}+ \frac{1}{R} [u-\tilde{P}(u,R)]^2_{H^1(Q^+_R )} \Bigg)\\
  & +C\Big( [w_1]^2_{H^1(\pa ' Q^+_\rho )} + \frac{1}{\rho}[w_1]^2_{H^1( Q^+_\rho )} + [w_2]^2_{H^1(\pa ' Q^+_\rho )}+ \frac{1}{\rho}[w_2]^2_{H^1( Q^+_\rho )} \Big)     \\
  & \le  C\left(\frac{\rho}{R}\right)^{n+2} \Bigg( [u-P(u,R)]^2_{H^1(\pa ' Q^+_R )}+ \frac{1}{R} [u-P(u,R)]^2_{H^1(Q^+_R )}  \\
 &+ [u-\tilde{P}(u,R)]^2_{H^1(\pa ' Q^+_R )}+ \frac{1}{R} [u-\tilde{P}(u,R)]^2_{H^1(Q^+_R )} \Bigg)\\
  &+ \frac{C}{R}  \sum_{i=1}^n \| F_i \|^2_{L^2( Q^+_R)} +C\| F_n \|^2_{L^2(\pa ' Q^+_R)} + C\| f \|^2_{L^2(\pa ' Q^+_R)}.
 \end{split}
\end{equation}
for any $\rho \in (0,R/2]$. The first calculation \eqref{eq:nonhomogeneous iteration computation1} proves the first inequality \eqref{eq:non-homogeneous-iteration-1} of Lemma \ref{lem:non-homogeneous-iteration}. To prove the second inequality \eqref{eq:non-homogeneous-iteration-2},  we fix the last term in the last line of  the calculation \eqref{eq:nonhomogeneous iteration computation2}. To achieve that, set 
$$
U:=u-P(u,R),
$$
then $U$ is a $H^{1}\left(\pa ' Q^+_R \right) \bigcap L^{\infty}\left(I_R:H^1(B_R^+)\right)$ CIE solution satisfying
\begin{equation*}
 \left\{
\begin{aligned}
  \pa_i(a_{ij}\pa_j U)  &=   \pa_i F_i \quad  \mbox{in }    Q^+_R,     \\
   \partial_{t}U-  a_{nj}\pa_jU  &=  f-(f)_R \quad  \mbox{on }  \pa ' Q^+_R .
\end{aligned}
\right.
\end{equation*}
Applying the calculation shown in \eqref{eq:nonhomogeneous iteration computation2} to $U$, 
\begin{equation*} 
  \begin{split}
          &[U-P(U,\rho)]^2_{H^1(\pa ' Q^+_\rho )}+\frac{1}{\rho} [U-P(U,\rho)]^2_{H^1(Q^+_\rho )}\\
          &+[U-\tilde{P}(U,\rho)]^2_{H^1(\pa ' Q^+_\rho )}+\frac{1}{\rho} [U-\tilde{P}(U,\rho)]^2_{H^1(Q^+_\rho )}\\
& \le\left(\frac{\rho}{R}\right)^{n+2} \Bigg( [U-P(U,R)]^2_{H^1(\pa ' Q^+_R )}+ \frac{1}{R} [U-P(U,R)]^2_{H^1(Q^+_R )} \\
&+    [U-\tilde{P}(U,R)]^2_{H^1(\pa ' Q^+_R )}+ \frac{1}{R} [U-\tilde{P}(U,R)]^2_{H^1(Q^+_R )}   \Bigg)\\
  &+ \frac{C}{R}  \sum_{i=1}^n \| F_i \|^2_{L^2( Q^+_R)} +C\| F_n \|^2_{L^2(\pa ' Q^+_R)} + C\| f-(f)_R \|^2_{L^2(\pa ' Q^+_R)}.
 \end{split}
\end{equation*}
\eqref{eq:non-homogeneous-iteration-2} is proved by noticing
$$
[U-P(U,\rho)]^2_{H^1(\pa ' Q^+_\rho )}=[u-P(u,\rho)]^2_{H^1( \pa ' Q^+_\rho )}, \  [U-P(U,\rho)]^2_{H^1( Q^+_\rho )}=[u-P(u,\rho)]^2_{H^1( Q^+_\rho )}
$$
and
$$
[U-\tilde{P}(U,\rho)]^2_{H^1(\pa ' Q^+_\rho )}=[u-\tilde{P}(u,\rho)]^2_{H^1(\pa ' Q^+_\rho )}, \ [U-\tilde{P}(U,\rho)]^2_{H^1( Q^+_\rho )}=[u-\tilde{P}(u,\rho)]^2_{H^1( Q^+_\rho )}.
$$
\end{proof}

Lemma \ref{lem:non-homogeneous-iteration} is used to apply iteration for Schauder estimate. We do not have the iteration on 
$$
[u]^2_{H^1(\pa ' Q^+_\rho )}
$$
solely. Instead, we successfully launch the iteration on
$$
[u]^2_{H^1(\pa ' Q^+_\rho )} + \frac{1}{\rho}[u]^2_{H^1(Q^+_\rho )}.
$$
In the next section, we freeze the coefficients to derive estimate of $[u]^2_{C^{1+\alpha}(\pa ' Q^+_{R/2} )}$ and $[u]^2_{C^{1+\alpha}(Q^+_{R/2} )}$.
\section{Schauder Estimates } \label{sec:Schauder}

We adapt Companato theory to get Schauder estimates to CIE solutions of \eqref{eq:boundary diffusion local standard}. As preparation, we recall the famous iteration lemma.
\begin{lem}\label{lem:schauder-iteration}
Let
 $$f:(0,R_0] \rightarrow \R $$
be a non negative, non decreasing function and for any $\rho, R :  0<\rho < R \le R_0$, there holds 
\begin{equation} \label{eq:schauder-iteration-1}
  \begin{split}
f(\rho) \le A\left( \left(\frac{\rho}{R}\right)^\alpha + \epsilon \right) f(R) + BR^{\beta}
 \end{split},
\end{equation}
where $A$,$B$,$\alpha$,$\beta$ are non negative constants and $0<\beta<\alpha$, then there exists a constant $\epsilon_0=\epsilon_0(A,\alpha, \beta)$ such that 
\begin{equation} \label{eq:schauder-iteration-2}
  \begin{split}
f(\rho) \le C(A,\alpha, \beta)\left( \left(\frac{\rho}{R}\right)^\beta  f(R)  + B\rho^{\beta}  \right),\quad  0<\rho <R\le R_0
 \end{split}
\end{equation}
whenever $\epsilon \le \epsilon_0 $.

\end{lem}
As another preparation, we define the Morrey spaces and Campanato spaces related to the local boundary diffusion equation \eqref{eq:boundary diffusion local standard}. For $p\ge 1, \theta\ge0$, we define Morrey spaces $L^{p,\theta}(\pa ' Q^+_R)$ and  $L^{p,\theta}( Q^+_R)$ with norms
$$
\| u \|_{L^{p,\theta}(\pa ' Q^+_R)}=\left( \sup_{  0<\rho<R } \frac{1}{\rho^\theta}  \int_{ \pa ' Q^+_\rho } |u(y,s)|^p dy'ds         \right)^{\frac{1}{p}}
$$
and 
$$
\| u \|_{L^{p,\theta}(\pa ' Q^+_R)}=\left( \sup_{  0<\rho<R } \frac{1}{\rho^{\theta+1}}  \int_{  Q^+_\rho } |u(y,s)|^p dyds         \right)^{\frac{1}{p}}.
$$
And we define Campanato spaces $\mathscr{L}^{p,\theta}(\pa ' Q^+_R)$ and $\mathscr{L}^{p,\theta}( Q^+_R)$ with norms
 $$
 \| u \|_{\mathscr{L}^{p,\theta}(\pa ' Q^+_R)} = \left(\| u \|^p_{L^{p}(\pa ' Q^+_R)} + [u ]^p_{\mathscr{L}^{p,\theta}(\pa ' Q^+_R)}\right)^{\frac{1}{p}}
 $$
and
  $$
 \| u \|_{\mathscr{L}^{p,\theta}(Q^+_R)} = \left(\| u \|^p_{L^{p}( Q^+_R)} + [u ]^p_{\mathscr{L}^{p,\theta}( Q^+_R)}\right)^{\frac{1}{p}}
 $$
 for functions in $L^{p}(\pa ' Q^+_R)$ and $L^{p}( Q^+_R)$ respectively. The semi-norms are written as
$$
[u ]_{\mathscr{L}^{p,\theta}(\pa ' Q^+_R)}=\left( \sup_{ 0<\rho<R} \frac{1}{\rho^\theta}  \int_{ \pa ' Q^+_\rho } |u(y,s)-u_{ \rho   }|^p dy'ds             \right)^{\frac{1}{p}},
$$
and 
$$
[u ]_{\mathscr{L}^{p,\theta}( Q^+_R)}=\left( \sup_{ 0<\rho<R} \frac{1}{\rho^{\theta+1}}  \int_{  Q^+_\rho } |u(y,s)-\tilde{u_{ \rho   }}|^p dyds             \right)^{\frac{1}{p}},
$$
where
$$
u_{\rho} = \fint_{\pa ' Q^+_\rho  }  u (y,s)dy'ds  , 
$$
and
$$
\tilde{u_{\rho}} = \fint_{ Q^+_\rho  }  u (y,s)dyds.
$$
To utilize Companato theory,  we recall that 
$$
   L^{p,\theta}(\pa ' Q^+_R)  \cong \mathscr{L}^{p,\theta}(\pa ' Q^+_R), \quad L^{p,\theta}( Q^+_R)  \cong \mathscr{L}^{p,\theta}( Q^+_R), \quad \text{for } 0 \le \theta <n,
$$
and 
$$
   C^{\alpha}(\pa ' Q^+_R)  \cong \mathscr{L}^{p,\theta}(\pa ' Q^+_R),\quad C^{\alpha}( Q^+_R)  \cong \mathscr{L}^{p,\theta}(Q^+_R)   , \quad \text{for } n < \theta \le n+p,
$$
with $\alpha= \frac{\theta -n}{p}$. Denoting 
\be\label{eq:morrey space seminorm H1}
[u]^2_{H^{1,\theta}(\pa ' Q^+_R )}= \sum_{j=1}^{n}\| \pa_{j} u \|^2_{L^{2,\theta}(\pa ' Q^+_R)} + \| \pa_{t} u \|^2_{L^{2,\theta}(\pa ' Q^+_R)},
\ee
\be\label{eq:morrey space seminorm H1-interior}
[u]^2_{H^{1,\theta}( Q^+_R )}= \sum_{j=1}^{n}\| \pa_{j} u \|^2_{L^{2,\theta}( Q^+_R)} + \| \pa_{t} u \|^2_{L^{2,\theta}( Q^+_R)},
\ee
and
\be\label{eq:campanato space seminorm H1}
[u]^2_{\mathscr{H}^{1,\theta}(\pa ' Q^+_R )}= \sum_{j=1}^{n}[ \pa_{j} u ]^2_{\mathscr{L}^{2,\theta}(\pa ' Q^+_R)} + [ \pa_{t} u ]^2_{\mathscr{L}^{2,\theta}(\pa ' Q^+_R)},
\ee
\be\label{eq:campanato space seminorm H1-interior}
[u]^2_{\mathscr{H}^{1,\theta}( Q^+_R )}= \sum_{j=1}^{n}[ \pa_{j} u ]^2_{\mathscr{L}^{2,\theta}( Q^+_R)} + [ \pa_{t} u ]^2_{\mathscr{L}^{2,\theta}(Q^+_R)},
\ee
then we have expressions 
$$
[u]_{\mathscr{H}^{1,n+2\alpha}(\pa ' Q^+_R )} \cong  [u]_{  C^{1+\alpha}(\pa ' Q^+_R)   },\quad [u]_{\mathscr{H}^{1,n+2\alpha}( Q^+_R )} \cong  [u]_{  C^{1+\alpha}( Q^+_R)   },
$$
for some $\alpha: 0<\alpha \le 1  $.

Now, We seek local Schauder estimates for solution to \eqref{eq:boundary diffusion local standard} with condition \eqref{eq:compatibility} and coefficients satisfying \eqref{eq:bound of coefficients-local} by applying Lemma \ref{lem:non-homogeneous-iteration} to such solution.

\begin{thm}\label{thm:estimate morrey norm}
If  $u\in    H^{1}\left(\pa ' Q^+_R \right) \bigcap L^{\infty}\left(I_R:H^1(B_R^+)\right) \bigcap L^{2,\theta}(\pa ' Q^+_R  ) $ is a CIE solution to \eqref{eq:boundary diffusion local standard}, with coefficients satisfying \eqref{eq:bound of coefficients-local} and $f\in L^{2,\theta}(\pa ' Q^+_R)(0\le \theta <n ) $, then 
$$
D_{x} u ,\pa_t u \in L^{2,\theta}(\pa ' Q^+_{R/2}), \quad D_{x} u ,\pa_t u \in L^{2,\theta}( Q^+_{R/2}),
$$
and 
\be\label{eq:estimate morrey norm}
\begin{split} 
& [u]^2_{H^{1,\theta}(\pa ' Q^+_{R/2} )}+[u]^2_{H^{1,\theta}( Q^+_{R/2} )} \\
&\le C \left( [u]^2_{H^1(\pa ' Q^+_R )} +\frac{1}{R}[u]^2_{H^1( Q^+_R )} + \| u \|^2_{ L^{2,\theta}(\pa ' Q^+_R)} + \| f \|^2_{ L^{2,\theta}(\pa ' Q^+_R)}  \right),
    \end{split}
\ee
where the positive constant $C$ is depending on $n,\Lambda,\theta$.
\end{thm}

\begin{proof}

To apple Lemma \ref{lem:non-homogeneous-iteration}, we write a rearranged version of  \eqref{eq:boundary diffusion local standard} as
\begin{equation}\label{eq:boundary diffusion local-freezing coefficients}
 \left\{
\begin{aligned}
  \pa_i(a(0,0)a_{ij}(0,0)\pa_j u)  &=   \pa_i F_i  \quad  \mbox{in }    Q^+_{R},     \\
   \partial_{t}u- a(0,0)\tilde{\varphi}(0)a_{nj}(0,0)\pa_ju   &=  \hat{f} \quad  \mbox{on }  \pa ' Q^+_{R} .
\end{aligned}
\right.
\end{equation} 
where
$$
F_i =  \sum_j \left[ a(0,0)a_{ij}(0,0)-a(x,t)a_{ij}(x,t) \right]\pa_{j}u, 
$$
$$
\hat{f}=\left[a(x,t)\tilde{\varphi}(x)a_{nj}(x,t)-a(0,0)\tilde{\varphi}(0)a_{nj}(0,0)\right]\pa_j u +f-bu,
$$
then 
$$
\int_{ \pa' Q^+_{R/2}  }  \hat{f}^2  \le C\left(   \int_{ \pa' Q^+_{R/2}  }  f^2 + \int_{ \pa' Q^+_{R/2}  }  u^2 +   R^{2\alpha}  [u]^2_{H^1(\pa ' Q^+_{R/2} )}        \right) 
$$
and
$$
 \sum_{i=1}^n\| F_i \|^2_{L^2( Q^+_{R/2})} \le C     R^{2\alpha}[u]^2_{H^1( Q^+_{R/2} )}, \quad \| F_n \|^2_{L^2( \pa ' Q^+_{R/2})} \le C     R^{2\alpha}[u]^2_{H^1( \pa' Q^+_{R/2} )} .     
$$
Employing Lemma \ref{lem:non-homogeneous-iteration} to solution $u$ of equation \eqref{eq:boundary diffusion local-freezing coefficients},
\begin{equation*} 
  \begin{split}
 &[u]^2_{H^1(\pa ' Q^+_\rho )} +\frac{1}{\rho}[u]^2_{H^1( Q^+_\rho )} \\
 &\le C\left(\frac{\rho}{R}\right)^{n} \left(  [u]^2_{H^1(\pa ' Q^+_{R/2} )} + \frac{1}{R}[u]^2_{H^1(Q^+_{R/2} )}\right)+ \frac{C}{R}\sum_{i=1}^n \| F_i \|^2_{L^2( Q^+_{R/2})}\\
 &+C \| F_n\|^2_{L^2( \pa' Q^+_{R/2})} + C\| \hat{f} \|^2_{L^2(\pa ' Q^+_{R/2})},
    \end{split}
\end{equation*}
for $\rho \in (0,R/2)$. 
Consequently,
\begin{equation*} 
  \begin{split}
 [u]^2_{H^1(\pa ' Q^+_\rho )} +\frac{1}{\rho}[u]^2_{H^1( Q^+_\rho )}&\le C \Bigg( \left[ \left(\frac{\rho}{R}\right)^{n}+     
  R^{2\alpha} \right]
  \left(  [u]^2_{H^1(\pa ' Q^+_{R/2} )} + \frac{1}{R}[u]^2_{H^1(Q^+_{R/2} )}\right) \\
  &+ R^\theta \left[\| f \|^2_{L^{2,\theta}(\pa ' Q^+_{R/2})}+  \| u \|^2_{L^{2,\theta}(\pa ' Q^+_{R/2})}   \right]  \Bigg),
    \end{split}
\end{equation*}
for $0<\rho<R/2$. By Lemma \ref{lem:schauder-iteration}, there is a constant $\epsilon_0=\epsilon_0(n,\Lambda,\theta )$ such that
\begin{equation}\label{eq:H^1-itertation}
\begin{split}
&[u]^2_{H^1(\pa ' Q^+_\rho )}+ \frac{1}{\rho}[u]^2_{H^1( Q^+_\rho )}\\
&  \le C \rho^{\theta} \left( \frac{1}{R_0^\theta}\left( [u]^2_{H^1(\pa ' Q^+_{R_0} )}+\frac{1}{R_0}[u]^2_{H^1(Q^+_{R_0} )}\right) +  \| f \|^2_{L^{2,\theta}(\pa ' Q^+_{R/2})}+  \| u \|^2_{L^{2,\theta}(\pa ' Q^+_{R/2})}      \right)
  \end{split}
\end{equation}
for $\rho \in (0,R_0]$, $R_0=\min\{ R/2,\epsilon_0^{1/2\alpha}  \}$.  For $\rho \in (R_0,R/2]$, 
\begin{equation}\label{eq:H^1-itertation-2}
\begin{split}
&[u]^2_{H^1(\pa ' Q^+_\rho )} +\frac{1}{\rho}[u]^2_{H^1(Q^+_\rho )} \\
 & \le C \left( \frac{\rho}{R_0}  \right)^\theta \left( [u]^2_{H^1(\pa ' Q^+_{R} )}+\frac{1}{R}[u]^2_{H^1(Q^+_{R} )}  +  \| f \|^2_{L^{2,\theta}(\pa ' Q^+_{R/2})}+  \| u \|^2_{L^{2,\theta}(\pa ' Q^+_{R/2})}      \right).
\end{split}
\end{equation}
According to notation \eqref{eq:morrey space seminorm H1} and the definition of the Morrey spaces, 
\begin{equation*}
\begin{split} 
& [u]^2_{H^{1,\theta}(\pa ' Q^+_{R/2} )}+[u]^2_{H^{1,\theta}( Q^+_{R/2} )} \\
&\le C \left( [u]^2_{H^1(\pa ' Q^+_R )} +\frac{1}{R}[u]^2_{H^1( Q^+_R )} + \| u \|^2_{ L^{2,\theta}(\pa ' Q^+_R)} + \| f \|^2_{ L^{2,\theta}(\pa ' Q^+_R)}  \right).
    \end{split}
\end{equation*}
\end{proof}
Using above estimates on the Morrey spaces, we derive the Schauer estimates.

\begin{thm}\label{thm:estimate campanto norm}If  $u\in    H^{1}\left(\pa ' Q^+_R \right) \bigcap L^{\infty}\left(I_R:H^1(B_R^+)\right) \bigcap C^{\alpha}(\pa ' Q^+_R  ) (0<\alpha<1)$  is a CIE solution to \eqref{eq:boundary diffusion local standard} with coefficients satisfying \eqref{eq:bound of coefficients-local} and $f\in C^{\alpha}(\pa ' Q^+_R) $, then 
$$
D_{x} u ,\pa_t u \in C^{\alpha}(\pa ' Q^+_{R/2}), \quad D_{x} u ,\pa_t u \in C^{\alpha}( Q^+_{R/2}),
$$
and 
\be\label{eq:estimate campanato norm}
[u]^2_{C^{1+\alpha}(\pa ' Q^+_{R/2} )} +[u]^2_{C^{1+\alpha}( Q^+_{R/2} )}\le C\left( [u]^2_{H^1(\pa ' Q^+_R )}+\frac{1}{R}[u]^2_{H^1(Q^+_R )} + \| u \|^2_{ C^{\alpha}(\pa ' Q^+_R)} + \| f \|^2_{ C^{\alpha}(\pa ' Q^+_R)}\right).
\ee
where the positive constant $C$ depends on $n,\Lambda,\theta$.
\end{thm}
\begin{proof}
As the proof of Theorem \ref{thm:estimate morrey norm}, we  rewrite equation \eqref{eq:boundary diffusion local standard} into \eqref{eq:boundary diffusion local-freezing coefficients} and utilize Lemma \ref{lem:non-homogeneous-iteration}, obtaining
\begin{equation*} 
  \begin{split}
         &[u-P(u,\rho)]^2_{H^1(\pa ' Q^+_\rho )}+\frac{1}{\rho}[u-P(u,\rho)]^2_{H^1(Q^+_\rho )}\\
         +&[u-\tilde{P}(u,\rho)]^2_{H^1(\pa ' Q^+_\rho )}+\frac{1}{\rho}[u-\tilde{P}(u,\rho)]^2_{H^1(Q^+_\rho )}\\
& \le  C\left(\frac{\rho}{R}\right)^{n+2} \Bigg( [u-P(u,R)]^2_{H^1(\pa ' Q^+_R )}+ \frac{1}{R} [u-P(u,R)]^2_{H^1(Q^+_R )}\\
&+[u-\tilde{P}(u,R)]^2_{H^1(\pa ' Q^+_R )}+ \frac{1}{R} [u-\tilde{P}(u,R)]^2_{H^1(Q^+_R )}\Bigg)\\
  &+  \frac{C}{R} \sum_{i=1}^n \| F_i \|^2_{L^2( Q^+_R)} +C\| F_n \|^2_{L^2(\pa ' Q^+_R)} + C\| f-(f)_R \|^2_{L^2(\pa ' Q^+_R)} 
 \end{split}
\end{equation*}
for $\rho \in (0,R]$. 

By the definition of $\hat{f}$, 
\begin{equation*} 
  \begin{split}
         \| \hat{f} -(\hat{f})_{R/2}\|^2_{L^2(\pa ' Q^+_{R/2})} &\le C \| \hat{f} -(f)_{R/2}\|^2_{L^2(\pa ' Q^+_{R/2})}  \\
       & \le C \left( \| f -(f)_{R/2}\|^2_{L^2(\pa ' Q^+_{R/2})}+ \|u\|^2_{L^2(\pa ' Q^+_{R/2})} +R^{2\alpha}  [u]^2_{H^1(\pa ' Q^+_{R/2} )} \right)  .
 \end{split}
\end{equation*}
Consequently,
\begin{equation} \label{eq:H^1-estimate-compare to polynomial-nonhomogeneous}
  \begin{split}
         &[u-P(u,\rho)]^2_{H^1(\pa ' Q^+_\rho )}+ \frac{1}{\rho}  [u-P(u,\rho)]^2_{H^1(Q^+_\rho )} \\
        & +[u-\tilde{P}(u,\rho)]^2_{H^1(\pa ' Q^+_\rho )}+\frac{1}{\rho}[u-\tilde{P}(u,\rho)]^2_{H^1(Q^+_\rho )}\\
         & \le  C \bigg[\left(\frac{\rho}{R}\right)^{n+2} \Bigg( [u-P(u,R/2)]^2_{H^1(\pa ' Q^+_{R/2} )}+  \frac{1}{R}[u-P(u,R/2)]^2_{H^1(Q^+_{R/2} )}\\
        & +[u-\tilde{P}(u,R/2)]^2_{H^1(\pa ' Q^+_{R/2} )}+\frac{1}{R}[u-\tilde{P}(u,R/2)]^2_{H^1(Q^+_{R/2} )}\Bigg) \\
         &  + R^{n+2\alpha}\| f\|^2_{\mathscr{L}^{2,n+2\alpha}(\pa ' Q^+_{R/2})}
          +  R^{n+2\alpha}\|u\|^2_{\mathscr{L}^{2,n+2\alpha}(\pa ' Q^+_{R/2})} \\
          &+R^{2\alpha}  \left( [u]^2_{H^1(\pa ' Q^+_{R/2} )}+ \frac{1}{R}    [u]^2_{H^1( Q^+_{R/2} )}  \right) \bigg] \\
      & \le  C \Bigg\{ \left(\frac{\rho}{R}\right)^{n+2} \Bigg( [u-P(u,R/2)]^2_{H^1(\pa ' Q^+_{R/2} )}+  \frac{1}{R}[u-P(u,R/2)]^2_{H^1(Q^+_{R/2} )}\\
        & +[u-\tilde{P}(u,R/2)]^2_{H^1(\pa ' Q^+_{R/2} )}+\frac{1}{R}[u-\tilde{P}(u,R/2)]^2_{H^1(Q^+_{R/2} )}\Bigg) \\
           & + R^{n+2\alpha}\| f\|^2_{C^{\alpha}(\pa ' Q^+_{R/2})}  +  R^{n+2\alpha}\|u\|^2_{C^{\alpha}(\pa ' Q^+_{R/2})} +R^{2\alpha}  \left( [u]^2_{H^1(\pa ' Q^+_{R/2} )}+ \frac{1}{R} [u]^2_{H^1( Q^+_{R/2} )}  \right) \Bigg\}
 \end{split}
\end{equation}
for $\rho \in (0,R/2]$, if we combine above two inequalities. By Theorem \ref{thm:estimate morrey norm}, 
\begin{equation}\label{eq:boudary diffusion schauder iteration1}
\begin{split}
&[u]^2_{H^1(\pa ' Q^+_{R/2} )}+ \frac{1}{R}[u]^2_{H^1( Q^+_{R/2} )}\\
& \le C R^{n-\alpha}\left( [u]^2_{H^1(\pa ' Q^+_{R} )}+\frac{1}{R}[u]^2_{H^1(Q^+_{R} )} +  \| f \|^2_{L^{2,n-\alpha}(\pa ' Q^+_{R})}+  \| u \|^2_{L^{2,n-\alpha}(\pa ' Q^+_{R})}     \right).
\end{split}
\end{equation}
Putting \eqref{eq:boudary diffusion schauder iteration1} into \eqref{eq:H^1-estimate-compare to polynomial-nonhomogeneous},  
\begin{equation} \label{eq:boudary diffusion schauder iteration2}
  \begin{split}
         &[u-P(u,\rho)]^2_{H^1(\pa ' Q^+_\rho )}+ \frac{1}{\rho}  [u-P(u,\rho)]^2_{H^1(Q^+_\rho )} \\
        & +[u-\tilde{P}(u,\rho)]^2_{H^1(\pa ' Q^+_\rho )}+\frac{1}{\rho}[u-\tilde{P}(u,\rho)]^2_{H^1(Q^+_\rho )}\\
         & \le  C \bigg[\left(\frac{\rho}{R}\right)^{n+2} \Bigg( [u-P(u,R/2)]^2_{H^1(\pa ' Q^+_{R/2} )}+  \frac{1}{R}[u-P(u,R/2)]^2_{H^1(Q^+_{R/2} )}\\
        & +[u-\tilde{P}(u,R/2)]^2_{H^1(\pa ' Q^+_{R/2} )}+\frac{1}{R}[u-\tilde{P}(u,R/2)]^2_{H^1(Q^+_{R/2} )}\Bigg) \\
          & + R^{n+\alpha} \Big( \| f\|^2_{\mathscr{L}^{2,n+2\alpha}(\pa ' Q^+_{R})}
         +  \|u\|^2_{\mathscr{L}^{2,n+2\alpha}(\pa ' Q^+_{R})} 
         +  [u]^2_{H^1(\pa ' Q^+_{R} )}  +  \frac{1}{R}  [u]^2_{H^1( Q^+_{R} )}  \Big) \bigg] \\
      & \le  C \bigg[\left(\frac{\rho}{R}\right)^{n+2} \Bigg( [u-P(u,R/2)]^2_{H^1(\pa ' Q^+_{R/2} )}+  \frac{1}{R}[u-P(u,R/2)]^2_{H^1(Q^+_{R/2} )}\\
        & +[u-\tilde{P}(u,R/2)]^2_{H^1(\pa ' Q^+_{R/2} )}+\frac{1}{R}[u-\tilde{P}(u,R/2)]^2_{H^1(Q^+_{R/2} )}\Bigg) \\
            &+ R^{n+\alpha} \Big( \| f\|^2_{C^{\alpha}(\pa ' Q^+_{R})} +  \|u\|^2_{C^{\alpha}(\pa ' Q^+_{R})} +  [u]^2_{H^1(\pa ' Q^+_{R} )}+ \frac{1}{R}  [u]^2_{H^1( Q^+_{R} )} \Big) \bigg]  
 \end{split}
\end{equation}
for $\rho \in (0,R/2]$.  Applying Lemma \ref{lem:schauder-iteration} to above inequality \eqref{eq:boudary diffusion schauder iteration2}, 
\begin{equation*} 
  \begin{split}
         &[u-P(u,\rho)]^2_{H^1(\pa ' Q^+_\rho )}  + \frac{1}{\rho}   [u-P(u,\rho)]^2_{H^1(Q^+_\rho )}        \\
         & +[u-\tilde{P}(u,\rho)]^2_{H^1(\pa ' Q^+_\rho )}+\frac{1}{\rho}[u-\tilde{P}(u,\rho)]^2_{H^1(Q^+_\rho )}\\
         & \le  C \rho^{n+\alpha} \Bigg\{
           \frac{1}{R^{n+\alpha}}   \Bigg( [u-P(u,R/2)]^2_{H^1(\pa ' Q^+_{R/2} )}+  \frac{1}{R}[u-P(u,R/2)]^2_{H^1(Q^+_{R/2} )}\\
        & +[u-\tilde{P}(u,R/2)]^2_{H^1(\pa ' Q^+_{R/2} )}+\frac{1}{R}[u-\tilde{P}(u,R/2)]^2_{H^1(Q^+_{R/2} )}\Bigg) \\
           &  +  \| f\|^2_{ C^{\alpha}(\pa ' Q^+_{R})}  +  \|u\|^2_{C^{\alpha}(\pa ' Q^+_{R})} +  [u]^2_{H^1(\pa ' Q^+_{R} )} + \frac{1}{R} [u]^2_{H^1(Q^+_{R} )} \Bigg\} \\
      & \le   C \rho^{n+\alpha} \bigg(
          [u]^2_{H^1(\pa ' Q^+_{R} )} +  \frac{1}{R}  [u]^2_{H^1(Q^+_{R} )} +  \| f\|^2_{ C^{\alpha}(\pa ' Q^+_{R})} +  \|u\|^2_{C^{\alpha}(\pa ' Q^+_{R})}     \bigg)  
 \end{split}
\end{equation*}
for $\rho \in (0,R/2]$.  Then
  \begin{equation*}
  \begin{split}
         &[u-P(u,\rho)]^2_{H^1(\pa ' Q^+_\rho )} +\frac{1}{\rho}[u-\tilde{P}(u,\rho)]^2_{H^1(Q^+_\rho )}\\
      & \le   C \rho^{n+\alpha} \bigg(
          [u]^2_{H^1(\pa ' Q^+_{R} )} +  \frac{1}{R}  [u]^2_{H^1(Q^+_{R} )} +  \| f\|^2_{ C^{\alpha}(\pa ' Q^+_{R})} +  \|u\|^2_{C^{\alpha}(\pa ' Q^+_{R})}     \bigg), 
 \end{split}
\end{equation*}
for $\rho \in (0,R/2]$.
By the definition of Campanato space, it provides
$$
[u]^2_{\mathscr{H}^{1,n+\alpha}(\pa ' Q^+_{R/2} )}  +   [u]^2_{\mathscr{H}^{1,n+\alpha}( Q^+_{R/2} )} \le   C  \bigg(
          [u]^2_{H^1(\pa ' Q^+_{R} )}+ \frac{1}{R}  [u]^2_{H^1(Q^+_{R} )}  +  \| f\|^2_{ C^{\alpha}(\pa ' Q^+_{R})} +  \|u\|^2_{C^{\alpha}(\pa ' Q^+_{R})}     \bigg). 
$$
As a result of the Campanto theory, 
$$
    \mathscr{L}^{2,n+\alpha}(\pa ' Q^+_{R/2})  \cong C^{\alpha/2}(\pa ' Q^+_{R/2}), \quad  \mathscr{L}^{2,n+\alpha}( Q^+_{R/2})  \cong C^{\alpha/2}( Q^+_{R/2})
$$
gives
$$
\pa_t u , D_x u \in C^{\alpha/2}(\pa ' Q^+_{R/2}), \quad \pa_t u , D_x u \in C^{\alpha/2}( Q^+_{R/2}).
$$
Correspondingly,
\begin{equation} \label{eq:H1 boudary estimate}
  \begin{split}
        [u]^2_{H^1(\pa ' Q^+_{R/2} )} &\le C |\pa ' Q^+_{R/2} |  \Big( | D_{x}u|^2_{C(\pa ' Q^+_{R/2})} + |\pa_t u|^2_{C(\pa ' Q^+_{R/2})} \Big)    \\
        &\le C R^{n}  [u]^2_{\mathscr{H}^{1,n+\alpha}(\pa ' Q^+_{R/2} )}\\
        &\le C  R^{n}  \bigg(
          [u]^2_{H^1(\pa ' Q^+_{R} )} +\frac{1}{R}[u]^2_{H^1( Q^+_{R} )} +  \| f\|^2_{ C^{\alpha}(\pa ' Q^+_{R})} +  \|u\|^2_{C^{\alpha}(\pa ' Q^+_{R})}     \bigg).
 \end{split}
\end{equation}
Similarly, 
\begin{equation} \label{eq:H1 interior estimate}
  \begin{split}
      \frac{1}{R}  [u]^2_{H^1(Q^+_{R/2} )} & \le \frac{C}{R}|Q^+_{R/2} |  \Big( | D_{x}u|^2_{C( Q^+_{R/2})} + |\pa_t u|^2_{C(Q^+_{R/2})} \Big)  \\
     &\le \frac{C}{R}  R^{n+1}  [u]^2_{\mathscr{H}^{1,n+\alpha}( Q^+_{R/2} )}\\
      &\le C  R^{n}  \bigg(
          [u]^2_{H^1(\pa ' Q^+_{R} )} +\frac{1}{R}[u]^2_{H^1( Q^+_{R} )} +  \| f\|^2_{ C^{\alpha}(\pa ' Q^+_{R})} +  \|u\|^2_{C^{\alpha}(\pa ' Q^+_{R})}     \bigg)    .
 \end{split}
\end{equation}
Plugging these two estimates \eqref{eq:H1 boudary estimate}-\eqref{eq:H1 interior estimate} into \eqref{eq:H^1-estimate-compare to polynomial-nonhomogeneous}, we have
\begin{equation} \label{eq:boudary diffusion schauder iteration3}
  \begin{split}
         &[u-P(u,\rho)]^2_{H^1(\pa ' Q^+_\rho )}  + \frac{1}{\rho}   [u-P(u,\rho)]^2_{H^1(Q^+_\rho )}        \\
         & +[u-\tilde{P}(u,\rho)]^2_{H^1(\pa ' Q^+_\rho )}+\frac{1}{\rho}[u-\tilde{P}(u,\rho)]^2_{H^1(Q^+_\rho )}\\
      & \le  C \bigg[\left(\frac{\rho}{R}\right)^{n+2}  \Bigg( [u-P(u,R/2)]^2_{H^1(\pa ' Q^+_{R/2} )}+  \frac{1}{R}[u-P(u,R/2)]^2_{H^1(Q^+_{R/2} )}\\
        & +[u-\tilde{P}(u,R/2)]^2_{H^1(\pa ' Q^+_{R/2} )}+\frac{1}{R}[u-\tilde{P}(u,R/2)]^2_{H^1(Q^+_{R/2} )}\Bigg) \\
           & + R^{n+2\alpha} \Big( \| f\|^2_{C^{\alpha}(\pa ' Q^+_{R/2})} +  \|u\|^2_{C^{\alpha}(\pa ' Q^+_{R/2})} +   [u]^2_{H^1(\pa ' Q^+_{R/2} )}+  \frac{1}{R} [u]^2_{H^1( Q^+_{R/2} )}  \Big) \bigg] .
 \end{split}
\end{equation}
for $\rho \in (0,R/2]$. Applying Lemma \ref{lem:schauder-iteration} again to \eqref{eq:boudary diffusion schauder iteration3}, 
 \begin{equation} \label{eq:boudary diffusion schauder iteration4}
  \begin{split}
         &[u-P(u,\rho)]^2_{H^1(\pa ' Q^+_\rho )}+[u-\tilde{P}(u,\rho)]^2_{H^1(Q^+_\rho )} \\
          & +[u-\tilde{P}(u,\rho)]^2_{H^1(\pa ' Q^+_\rho )}+\frac{1}{\rho}[u-\tilde{P}(u,\rho)]^2_{H^1(Q^+_\rho )}\\
         & \le  C\rho^{n+2\alpha} \Bigg\{\left(\frac{1}{R}\right)^{n+2\alpha}\bigg( [u-P(u,R/2)]^2_{H^1(\pa ' Q^+_{R/2} )}+  \frac{1}{R}[u-P(u,R/2)]^2_{H^1(Q^+_{R/2} )}\\
        & +[u-\tilde{P}(u,R/2)]^2_{H^1(\pa ' Q^+_{R/2} )}+\frac{1}{R}[u-\tilde{P}(u,R/2)]^2_{H^1(Q^+_{R/2} )}\bigg) \\
           &   \| f\|^2_{C^{\alpha}(\pa ' Q^+_{R/2})} +  \|u\|^2_{C^{\alpha}(\pa ' Q^+_{R/2})}
           +   [u]^2_{H^1(\pa ' Q^+_{R/2} )}+ \frac{1}{R}[u]^2_{H^1( Q^+_{R/2} )}   \Bigg\}  \\
            & \le  C \rho^{n+2\alpha} \bigg(
           [u]^2_{H^1(\pa ' Q^+_{R} )}+\frac{1}{R}[u]^2_{H^1(Q^+_{R} )} + \| f\|^2_{C^{\alpha}(\pa ' Q^+_{R})}  
            +  \|u\|^2_{C^{\alpha}(\pa ' Q^+_{R})}   \bigg)
 \end{split}
\end{equation}
for $\rho \in (0,R/2]$, so we have
\begin{equation} \label{eq:boudary diffusion schauder iteration5}
  \begin{split}
         &[u-P(u,\rho)]^2_{H^1(\pa ' Q^+_\rho )} +\frac{1}{\rho}[u-\tilde{P}(u,\rho)]^2_{H^1(Q^+_\rho )}\\
            & \le  C \rho^{n+2\alpha} \bigg(
           [u]^2_{H^1(\pa ' Q^+_{R} )}+\frac{1}{R}[u]^2_{H^1(Q^+_{R} )} + \| f\|^2_{C^{\alpha}(\pa ' Q^+_{R})}  
            +  \|u\|^2_{C^{\alpha}(\pa ' Q^+_{R})}   \bigg).
 \end{split}
\end{equation}  \eqref{eq:boudary diffusion schauder iteration5} implies 
$$
[u]_{\mathscr{H}^{1,n+2\alpha}(\pa ' Q^+_{R/2} )}+[u]_{\mathscr{H}^{1,n+2\alpha}(Q^+_{R/2} )} \le  C  \bigg(
           [u]^2_{H^1(\pa ' Q^+_{R} )}+\frac{1}{R}[u]^2_{H^1(Q^+_{R} )} + \| f\|^2_{C^{\alpha}(\pa ' Q^+_{R})}  
            +  \|u\|^2_{C^{\alpha}(\pa ' Q^+_{R})}   \bigg).
$$
Theorem \ref{thm:estimate campanto norm} is proved according to the Campanato theory.

\end{proof}
In addition, the above arguments also work for the  linear equation localized from equation \eqref{eq:nonlinear boundary diffusion}, which is written as
\begin{equation}\label{eq:boundary diffusion local standard-2}
 \left\{
\begin{aligned}
  \pa_i(a_{ij}(x,t)\pa_j u)  &=   0  \quad  \mbox{in }    Q^+_R= B^+_R \times I_R,     \\
   a(x,t)\partial_{t}u-  \tilde{\varphi}(x)  a_{nj}(x,t)\pa_ju + bu &=  f \quad  \mbox{on }  \pa ' Q^+_R =  \pa ' B^+_R \times I_R.
\end{aligned}
\right.
\end{equation}
The Schauder estimates of such type of boundary diffusion equation are obtained as well.

\begin{thm}\label{thm:schauder estimate local general-2}
If $u \in C^{1+\alpha}( \pa ' Q^+_R ) \bigcap C^{1+\alpha}(  Q^+_R  ) $ is a CIE solution to \eqref{eq:boundary diffusion local standard-2}  with coefficients satisfying \eqref{eq:bound of coefficients-local}, then we have Schauder type estimate
\be\label{eq:schauder estimate local general-2}
\begin{split}
&[u]^2_{C^{1+\alpha}(\pa ' Q^+_{R/2} )} + [u]^2_{C^{1+\alpha}( Q^+_{R/2} )} \\
&\le C(n,\Lambda,\alpha)\left( \|u\|^2_{H^1(\pa ' Q^+_R )}+\|u\|^2_{H^1( Q^+_R )} + \| u \|^2_{ C^{\alpha}(\pa ' Q^+_R )} + \| f \|^2_{ C^{\alpha}(\pa ' Q^+_R )}\right).
\end{split}
\ee

\end{thm}

\section{Applications } \label{sec:application}

This section exhibits several typical applications to the theory of Schauder estimates in Theorems \ref{thm:schauder estimate local general} and \ref{thm:schauder estimate global}.
As a direct application, the $C^{1+\alpha}$-type estimates for positive classical CIE solution of non-linear boundary diffusion equation \eqref{eq:nonlinear boundary diffusion} are generated.
\begin{thm}\label{thm:nonlinear boundary diffusion schauder estimate global}
If $u \in C^{1+\alpha}(\Omega \times (0,T)) \cap  C^{1+\alpha}(\pa \Omega\times (0,T))$ is a positive CIE solution to \eqref{eq:nonlinear boundary diffusion} for every local equation in $Q_R^+(x,t), (x,t) \in \pa \om \times (0,T)$, and it satisfies \eqref{eq:positive bounds} uniformly, then we have Schauder type estimate
\be\label{eq:nonlinear boundary diffusion schauder estimate global}
\begin{split}
&[u]^2_{C^{1+\alpha}(\pa \Omega \times (0,T) )}+[u]^2_{C^{1+\alpha}( \Omega \times (0,T) )} \\
&\le C(n,\Lambda,\alpha,p)\left( \|u\|^2_{H^1(\pa \Omega \times (0,T) )}+\|u\|^2_{H^1(\Omega \times (0,T)  )} + \| u \|^2_{ C^{\alpha}( \pa \Omega \times (0,T) )}  \right).
\end{split}
\ee

\end{thm}
\begin{rem}\label{rem:potive upper and lower bounds}
The assumption that the positive solution $u$ of non-linear diffusion problem \eqref{eq:nonlinear boundary diffusion} is bounded above and below uniformly, was essentially used to prove the short time existence result, see \cite{brendle2002generalization} and \cite{2024JXY}. Therefore we derive Schauder type a prior estimate \eqref{eq:nonlinear boundary diffusion schauder estimate global} in Theorem \ref{thm:nonlinear boundary diffusion schauder estimate global} under positive upper and lower bounds \eqref{eq:positive bounds}.

\end{rem}
With the Schauder type estimates \eqref{eq:schauder estimate global} of the linear equation \eqref{eq:boundary diffusion}, we get the existence result of $C^{1+\alpha}$ type CIE solution.
\begin{thm}\label{thm: boundary diffusion non-homogeneous existence}
The linear boundary diffusion equation \eqref{eq:boundary diffusion} with initial condition
\be \label{eq: boundary diffusion initial condition}
 v  =   0\quad  \mbox{on }     \partial \Omega\times \{ t=0 \}
\ee
and the coefficients satisfying \eqref{eq:bound of coefficients},
has a unique CIE solution 
$v \in C^{1+\alpha}(\Omega \times (0,T)) \cap  C^{1+\alpha}(\pa \Omega\times (0,T))$.

\end{thm}

\begin{proof}
Taking 
$$
L_0 v= \pa_t v + \pa_\nu v + bv,
$$
$$
L_1 v = \pa_t v + a(x,t)\pa_\nu v + bv,
$$
subsequently, we look into equations of family $\lambda \in [0,1]$
\begin{equation}\label{eq:boundary diffusion of a class lambda}
 \left\{
\begin{aligned}
  \mbox{div}(  [(1-\lambda) + \lambda  a    ]     \nabla v)  &=   0  \quad  \mbox{in }     \Omega\times (0,T),     \\
   L_\lambda v  &=  f    \quad  \mbox{on }     \partial \Omega\times (0,T),
\end{aligned}
\right.
\end{equation}
where $L_\lambda v:=(1-\lambda)L_0 v  +\lambda L_1 v$. Then from Theorem \ref{thm:schauder estimate global}, the existence is proved by the standard continuity method. The uniqueness is proved by the Comparison principle, see Proposition \textcolor{blue}{3.3} \cite{2024JXY}.
\end{proof}
 
Similarly, we have the existence result for a related linear equation as follow.
 
 \begin{thm}\label{thm:linear boundary diffusion non-homogeneous existence-u}
The linear boundary diffusion problem
\begin{equation}\label{eq:linear boundary diffusion non-homogeneous-u}
 \left\{
\begin{aligned}
  \Delta u  &=   0  \quad  \mbox{in }     \Omega\times (0,\infty),     \\
   a(x,t)\partial_{t}u + \partial_{\nu}u +  b u  &=    f   \quad  \mbox{on }     \partial \Omega\times (0,\infty),\\
   u  &=    0 \quad  \mbox{on }     \partial \Omega\times \{ t=0 \},
\end{aligned}
\right.
\end{equation}
with the coefficients' conditions \eqref{eq:bound of coefficients},
has a unique CIE solution 
$u \in C^{1+\alpha}(\Omega \times (0,T)) \cap  C^{1+\alpha}(\pa \Omega\times (0,T))$.

\end{thm}

Furthermore, by the Schauder type estimate of the linear equation \eqref{eq:boundary diffusion}, we also obtain short time existence of $C^{1+\alpha}$-type CIE solution for the non-linear boundary diffusion problem which  generates the problem \eqref{eq:nonlinear boundary diffusion}-\eqref{eq: boundary diffusion initial condition u_0} by scaling on the time variable, see \textcolor{blue}{(13)-(14)} of  \cite{2024JXY}, with non-homogeneous boundary term $f$. It is formulated as
\begin{equation}\label{eq:nonlinear boundary diffusion non-homogeneous-u}
 \left\{
\begin{aligned}
  \Delta u  &=   0  \quad  \mbox{in }     \Omega\times (0,\infty),     \\
   \partial_{t}u^{p}&=   -  \partial_{\nu}u  -  b u +f   \quad  \mbox{on }     \partial \Omega\times (0,\infty),\\
   u  &=   u_0 > 0 \quad  \mbox{on }     \partial \Omega\times \{ t=0 \},
\end{aligned}
\right.
\end{equation}
where $u_0 \in C^{1+\alpha}(\pa \om)$. $b(x,t)$ and $f(x,t)$ are given as in \eqref{eq:boundary diffusion}.
\begin{thm}\label{thm:nonlinear boundary diffusion non-homogeneous short time existence }
The non-linear boundary diffusion problem \eqref{eq:nonlinear boundary diffusion non-homogeneous-u} has a unique CIE solution 
$u\in C^{1+\alpha}(\Omega \times (0,T)) \cap  C^{1+\alpha}(\pa \Omega\times (0,T))$ on a small time interval $(0,T)$.

\end{thm}

\begin{proof}
To employ the Leray-Schauder theory, we assume  
$$\|u\|_{  C^{\alpha}(\bar{\om}\times (0,T)} + \|u\|_{  C^{\alpha}(\pa {\om}\times (0,T)} \le M,
$$
for some  constant $M>0$. In terms of the initial value $u_0 $, since $\pa \om$ is compact, one may assume
$$
\frac{2}{\Lambda} \le u_0 \le \frac{\Lambda}{2},
$$
where $\Lambda >>1$ is a positive constant. For $u_0 \in C^{1+\alpha}(\pa \om)$, there is an extension $U_0 \in C^{3}({ \om})\cap C^{1+\alpha}{(\pa \Omega)}$ such that 
\begin{equation}\label{eq:estension of initial value u_0}
 \left\{
\begin{aligned}
  \Delta U_0  &=   0  \quad  \mbox{in }     \Omega,    \\
   U_0  &=  u_0    \quad  \mbox{on }     \partial \Omega,
\end{aligned}
\right.
\end{equation}
then we build an equation 
\begin{equation}\label{eq:nonlinear boundary diffusion non-homogeneous-linearized}
 \left\{
\begin{aligned}
  \Delta \psi  &=   0  \quad  \mbox{in }     \Omega\times (0,\infty),     \\
   p(w+U_0)^{p-1}\partial_{t}\psi&=   -  \partial_{\nu}\psi -b \psi + \left(f - \partial_{\nu}U_0 -  b U_0 \right)   \quad  \mbox{on }     \partial \Omega\times (0,\infty),\\
   \psi  &=    0 \quad  \mbox{on }     \partial \Omega\times \{ t=0 \}.
\end{aligned}
\right.
\end{equation}
Denoting spaces
\be \label{eq:X space}
 X:=\left\{ w\in C^{\alpha} (\om \times (0,T) ) \cap C^{\alpha} (\pa \om \times (0,T)) \mid   w(x,0)=0    \right\},
\ee
and
\be \label{eq:Y space}
 Y:=\left\{ \psi\in C^{1+\alpha}(\Omega \times (0,T)) \cap  C^{1+\alpha}(\pa {\Omega} \times (0,T))  \mid  \psi(x,0)  = 0    \right\},
\ee
then for any $w \in X$,   
$$
\frac{1}{\Lambda} \le w+U_0 \le \Lambda,
$$
provided $T>0$ sufficiently small, such that
 $$
 | w(x,t) | \le \|w\|_{    C^{\alpha} (\bar{\om}\times (0,T) )   } T^{\alpha}  \le \frac{1}{2\Lambda} . 
 $$ 
Hence one finds a unique solution $\psi\in C^{1+\alpha}(\Omega \times (0,T)) \cap  C^{1+\alpha}(\pa \Omega\times (0,T))$ of problem \eqref{eq:nonlinear boundary diffusion non-homogeneous-linearized} by employing Theorem \ref{thm:linear boundary diffusion non-homogeneous existence-u}. 
Subsequently, it is proper to define a map
\begin{equation} \label{eq:map F}
  \begin{split}
         \mathsf{F}:X & \longrightarrow Y ,\\
                    w & \longmapsto \mathsf{F}(w)=\psi.
 \end{split}
\end{equation}
Here, it presents that $\phi:=\psi+U_0 \in C^{1+\alpha}(\Omega \times (0,T)) \cap  C^{1+\alpha}(\pa \Omega\times (0,T))$ satisfies
\begin{equation}\label{eq:nonlinear boundary diffusion non-homogeneous-linearized-with initial value}
 \left\{
\begin{aligned}
  \Delta \phi  &=   0  \quad  \mbox{in }     \Omega\times (0,\infty),     \\
   p(w+U_0)^{p-1}\partial_{t}\phi&=   -  \partial_{\nu}\phi -  b \phi  +f   \quad  \mbox{on }     \partial \Omega\times (0,\infty),\\
   \phi  &=    u_0  \quad  \mbox{on }     \partial \Omega\times \{ t=0 \}.
\end{aligned}
\right.
\end{equation}
In order to find a fixed point of $\mathsf{F}$, we embed \eqref{eq:nonlinear boundary diffusion non-homogeneous-linearized} into a family of problems
\begin{equation}\label{eq:nonlinear boundary diffusion non-homogeneous-linearized-sigma family}
 \left\{
\begin{aligned}
  \Delta \psi  &=   0  \quad  \mbox{in }     \Omega\times (0,\infty),     \\
\left(   \sigma p(w+U_0)^{p-1}\ +(1-\sigma) \right)  \partial_{t}\psi&=   -  \partial_{\nu}\psi  -\sigma \left( \partial_{\nu}U_0 + b (\psi +U_0)  -f \right)  \quad  \mbox{on }     \partial \Omega\times (0,\infty),\\
   \psi  &=    0 \quad  \mbox{on }     \partial \Omega\times \{ t=0 \},
\end{aligned}
\right.
\end{equation}
with $\sigma \in [0,1]$. The family \eqref{eq:nonlinear boundary diffusion non-homogeneous-linearized-sigma family} keeps the structure
$$
\frac{1}{\Lambda}  \le \sigma p(w+U_0)^{p-1}\ +(1-\sigma) \le \Lambda,
$$ 
uniformly for all $\sigma \in [0,1] $ provided $T>0$  small enough. One finds a unique solution $\psi_\sigma \in C^{1+\alpha}(\Omega \times (0,T)) \cap  C^{1+\alpha}(\pa \Omega\times (0,T))$ of problem \eqref{eq:nonlinear boundary diffusion non-homogeneous-linearized-sigma family} for every  $\sigma \in [0,1]$, and another map is defined as follow
 \begin{equation} \label{eq:map T}
  \begin{split}
         \mathsf{T}:X\times [0,1] & \longrightarrow Y ,\\
                    (w,\sigma) & \longmapsto \mathsf{T}(w,\sigma) = \psi_\sigma.
 \end{split}
\end{equation}
Such map $\mathsf{T}$ possesses the following properties:
\begin{enumerate}
\item[(a).]  $\mathsf{T}$ is compact, by that $Y\subset \subset X$. 
\item [(b).]  $\mathsf{T}(w,0)=0, \quad \forall w \in X$, since the problem
  \begin{equation}\label{eq:nonlinear boundary diffusion non-homogeneous-linearized-sigma family-zero}
 \left\{
\begin{aligned}
  \Delta \psi  &=   0  \quad  \mbox{in }     \Omega\times (0,\infty),     \\
  \partial_{t}\psi&=   -  \partial_{\nu}\psi    \quad  \mbox{on }     \partial \Omega\times (0,\infty),\\
   \psi  &=    0 \quad  \mbox{on }     \partial \Omega\times \{ t=0 \},
\end{aligned}
\right.
\end{equation}
only admits trivial solution.
\item [(c).] For any possible fixed point of $\mathsf{T}$, the $X$ norm of such fixed point is uniformly bounded. i.e., there is a positive constant $M$ such that, if $w \in X$ satisfies $w=T(w,\sigma)$ for some $\sigma \in [0,1]$, then $\|w\|_X \le C  M$. It is guaranteed by the assumption of H\"older estimate of $u$  that $\|u\|_{    C^{\alpha} (\bar{\om}\times (0,T) )   }  + \|u\|_{  C^{\alpha}(\pa {\om}\times (0,T)}\le M $.
\end{enumerate}
Now, the Leray-Schauder fixed point theorem tells that $\mathsf{T}(\cdot,1)$ must possess a fixed point $v$. i.e., there is some $v\in X$, such that $\mathsf{T}(v,1)=v$, i.e. $v \in C^{1+\alpha}(\Omega \times (0,T)) \cap  C^{1+\alpha}(\pa \Omega\times (0,T))$ is a solution to 
\begin{equation}\label{eq:nonlinear boundary diffusion non-homogeneous-linearized-sigma family-one}
 \left\{
\begin{aligned}
  \Delta v  &=   0  \quad  \mbox{in }     \Omega\times (0,\infty),     \\
 p(v+U_0)^{p-1}  \partial_{t}v &=   -  \partial_{\nu}(v+U_0) -  b (v+U_0)  +f  \quad  \mbox{on }     \partial \Omega\times (0,\infty),\\
   v  &=    0 \quad  \mbox{on }     \partial \Omega\times \{ t=0 \},
\end{aligned}
\right.
\end{equation}

Taking $u=v+U_0$, then  $u \in C^{1+\alpha}(\Omega \times (0,T)) \cap  C^{1+\alpha}(\pa \Omega\times (0,T))$ is a solution to \eqref{eq:nonlinear boundary diffusion non-homogeneous-u}. The uniqueness is proved by the Comparison principle as mentioned in Theorem \ref{thm: boundary diffusion non-homogeneous existence}.

\end{proof}
From Theorem \ref{thm:nonlinear boundary diffusion non-homogeneous short time existence }, one gets the following short time existence results.

\begin{cor}\label{cor:thm:nonlinear boundary diffusion  short time existence }

The non-linear boundary diffusion problem \eqref{eq:nonlinear boundary diffusion}-\eqref{eq: boundary diffusion initial condition u_0} has a unique CIE solution $u \in C^\infty(\om \times (0,T)) \bigcap C^\infty (\pa \om \times (0,T))$ on a small time interval $(0,T)$ if $\om$ is a smooth domain.
\end{cor}

As another consequence of estimate \eqref{eq:schauder estimate global} under condition \eqref{eq:bound of coefficients}, one can also obtain the existence result of a certain quasi-linear problem
\begin{equation}\label{eq:boundary diffusion nonlinear 1}
 \left\{
\begin{aligned}
  \mbox{div} (a(u)\nabla u)  &=   0  \quad  \mbox{in }     \Omega\times (0,T),     \\
   \partial_{t}u+a(u) \partial_{\nu}u + bu &=  f    \quad  \mbox{on }     \partial \Omega\times (0,T),\\
   u &= u_0 \in C^{1+\alpha}(\pa \om) \quad  \mbox{on } \partial \Omega\times \{ t=0 \},
\end{aligned}
\right.
\end{equation}
where $a(u)=(u^2+1)^{l/2}$, $l>0$, and $b,f \in C^{\alpha}(\partial \Omega\times (0,T) ) $. 
\begin{cor}\label{cor:schauder estiamte of certain nonlinear problem }

Suppose the solution of the problem \eqref{eq:boundary diffusion nonlinear 1} satisfies 
$$\|u\|_{    C^{\alpha} (\bar{\om}\times (0,T) )   } + \|u\|_{    C^{\alpha} (\pa {\om}\times (0,T) )   } \le M $$ 
for some $M>0$, then \eqref{eq:boundary diffusion nonlinear 1} admits a CIE solution  $u \in C^{1+\alpha}(\Omega \times (0,T)) \cap  C^{1+\alpha}(\pa \Omega\times (0,T))$.
\end{cor}

\begin{rem}\label{rem:assumption of Holder estimate}

To discuss the solution of problem \eqref{eq:nonlinear boundary diffusion non-homogeneous-u}, we need to use the assumption that $\|u\|_{    C^{\alpha} (\bar{\om}\times (0,T) )   }  + \|u\|_{    C^{\alpha} (\pa {\om}\times (0,T) )   }   \le M $. Such assumption can be obtained from the Harnack inequalities of corresponding weak solutions if the exponent $p$ is less than the Sobolev critical exponent. For $0<p<1, \ f \equiv 0$, Athanasopoulos-Caffarelli \cite{2010Continuity} proved continuity of the solution with the H\"older modulus in 2010. For $1<p<\frac{n}{n-2}, f \equiv 0$, Jin-Xiong-Yang obtained weak Harnack inequalities, see \cite{2024JXY} 2024. Similar Harnack inequalities have been proved in 2010 for the classical fast diffusion equation, see Bonforte- V\'azquez \cite{BV2010}.

Complete weak theory of boundary diffusion problems like \eqref{eq:nonlinear boundary diffusion non-homogeneous-u} is expected to be established in the future and it is very important in the research of boundary control problems in different initial-boundary settings. These are beyond the scope of the present paper.

\end{rem}

\bigskip

\noindent\textbf{Declarations of interest}: none. 

\bigskip

\noindent \textbf{Acknowledgement:} 
The author would like to thank Professor Tianling Jin of the Math Department, HKUST , who is my postdoc supervisor, for his valuable support and  advice. The author is currently visiting the math department of South China University of Technology while drafting this article. The author would like to thank Professor Zhengrong Liu for his kind help.

\bigskip

\noindent X. Yang

\noindent Email: \textsf{2768620240@qq.com}


\end{document}